\documentclass[lfeqn,12pt]{article}
\usepackage{amssymb,bbm,pifont,mathtools,mathabx,mathrsfs,
}
\usepackage{minitoc}
 
\usepackage[nottoc]{tocbibind}

\usepackage{comment}
  \usepackage{ulem}
\usepackage{multirow}
\usepackage[multiple]{footmisc}
\usepackage{xcolor}
\usepackage[colorlinks=true]{hyperref}
\usepackage{marginnote}

\usepackage{enumerate}

\newcommand{\bmat}[1]{\begin{bmatrix}#1\end{bmatrix}}

\newcommand{\Giu}{{\bigskip\noindent}}

\newcommand{\noi}{{\noindent}}

\newtheorem{theorem}{Theorem}
\newtheorem{definition}[theorem]{Definition}
\newtheorem{proposition}[theorem]{Proposition}
\newtheorem{lemma}[theorem]{Lemma}
\newtheorem{example}{Example}
\newtheorem{remark}[theorem]{Remark}
\newtheorem{coment}[theorem]{Comment}
\newtheorem{convent}[theorem]{Convention}
\newtheorem{conjecture}[theorem]{Conjecture}
\newtheorem{question}{Question}
\newtheorem{sublemma}[theorem]{Sublemma}
\newtheorem{corollary}[theorem]{Corollary}
\newtheorem{assumption}[theorem]{Assumption}
\newtheorem{notation}[theorem]{Notations}
\newtheorem{notationalrem}[theorem]{Notational Remark}

\newtheorem{tools}[subsection]{$\negsp\negsp$}
\newcommand\asm[1]{ \begin{assumption}\label{#1} }
\newcommand\easm{ \end{assumption} }
\newcommand\dfn[1]{ \begin{definition}\label{#1} }
\newcommand\dfntwo[2]{ \begin{definition}[#2]\label{#1} }
\newcommand\edfn{ \end{definition} }
\newcommand\examp[1]{ \begin{example}\label{#1} \small \rm}
\newcommand\rem[1]{ \begin{remark}\label{#1} \small \rm}
\newcommand\comt[1]{ \begin{coment}\label{#1} \small \rm}
\newcommand\convt[1]{ \begin{convent}\label{#1} \small \rm}
\newcommand\remtwo[2]{ \begin{remark}[#2]\label{#1} \rm}
\newcommand\eexamp{ \end{example} }
\newcommand\erem{ \end{remark} }
\newcommand\ecomt{ \end{coment} }
\newcommand\econvt{ \end{convent} }
\newcommand\thm[1]{ \begin{theorem}\label{#1}}
\newcommand\thmtwo[2]{ \begin{theorem}[#2]\label{#1}}
\newcommand\ethm{ \end{theorem} }
\newcommand\pro[1]{ \begin{proposition}\label{#1}}       
\newcommand\protwo[2]{ \begin{proposition}[#2]\label{#1}}
\newcommand\epro{ \end{proposition} }
\newcommand\conj[1]{ \begin{conjecture}\label{#1}}       
\newcommand\conjtwo[2]{ \begin{conjecture}[#2]\label{#1}}
\newcommand\econj{ \end{conjecture} }
\newcommand\quest[1]{ \begin{question}\label{#1}}       
\newcommand\questtwo[2]{ \begin{question}[#2]\label{#1}}
\newcommand\equest{ \end{question} }
\newcommand\lem[1]{ \begin{lemma}\label{#1}}
\newcommand\lemtwo[2]{ \begin{lemma}[#2]\label{#1}}
\newcommand\elem{ \end{lemma} }
\newcommand\sublem[1]{ \begin{sublemma}\label{#1}}
\newcommand\sublemtwo[2]{ \begin{sublemma}[#2]\label{#1}}
\newcommand\esublem{ \end{sublemma} }
\newcommand\cor[1]{ \begin{corollary}\label{#1}}
\newcommand\cortwo[2]{ \begin{corollary}[#2]\label{#1}}
\newcommand\ecor{ \end{corollary} }
\newcommand\notat[1]{ \begin{notation}\label{#1} \sl}
\newcommand\enotat{ \end{notation} }
\newcommand\notrem[1]{ \begin{notationalrem}\label{#1} \sl}
\newcommand\enotrem{ \end{notationalrem} }

\newcommand\equ[1]{{\rm (\ref{#1})}}
\newcommand\beq[1]{ \begin{equation}\label{#1} }
\newcommand{\eeq}{ \end{equation} }

\newcommand\beqa[1]{ \begin{eqnarray} \label{#1}}
\newcommand{\eeqa}{ \end{eqnarray} }
\newcommand{\beqano}{ \begin{eqnarray*} }
\newcommand{\eeqano}{ \end{eqnarray*} }

\newcommand{\proof}{\par\medskip\noindent{\bf Proof\ }}

\newcommand{\proofff[1]}{\par\medskip\noindent{\bf {#1}\ }}

\newcommand\warning[1]{{#1}}
\newcommand\rwarning[1]{{#1}}
\newcommand\gwarning[1]{{#1}}
\newcommand{\ie}{{\it i.e.\  }}
\newcommand{\fig}{{\it fig.}}

\newcommand{\etc}{{\it etc\ }}
\newcommand{\wrt}{{\it w.r.t.\ }}

\newcommand{\st}{{\it s.t.\ }}
\newcommand{\resp}{{\it resp.\ }}
\newcommand{\cf}{{\it cf.\ }}
\newcommand{\sss}{{\it iff\ }}

\newcommand{\qed}{\hskip.5truecm
            \vrule width 1.7truemm height 3.5truemm depth 0.truemm
            \par\Giu}
\newcommand{\qedeq}{\hskip.5truecm
            \vrule width 1.7truemm height 3.5truemm depth 0.truemm}

\newcommand\ovl[1]{ \overline {#1} }
\newcommand\uvl[1]{ \underline {#1} }

\newcommand\su[1]{ \frac{1}{ {#1}} }

\DeclareMathOperator\dist{dist}
\DeclareMathOperator\vdist{vdist}

\newcommand{\dpr}{ {\partial}   }

\newcommand\eqby[1]{\stackrel{\equ{#1}}{=}}

\newcommand{\negsp}{\hspace{-.04truecm}}
\newcommand\ex{\, e}

\newcommand{\g}{ {\gamma}   }
\newcommand{\G}{ {\Gamma}   }
\renewcommand{\d}{ {\delta}   }

\newcommand{\vae }{ {\varepsilon}   }

\renewcommand{\th }{ {\theta}   }

\renewcommand{\l}{ {\lambda}   }

\newcommand{\m}{ {\mu}   }

\newcommand{\f}{ {\varphi}   }

\renewcommand{\O}{ {\Omega}   }
\newcommand{\torus}{ {\mathbb{ T}}   }
\renewcommand{\natural}{ {\mathbb{ N}}   }
\newcommand{\real}{ {\mathbb{ R}}   }
\newcommand{\integer}{ {\mathbb{ Z}}   }

\newcommand{\rational}{ {\mathbb{ Q}}  }

\newcommand{\rd}{ {\real^2}   }

\newcommand{\el}{ {\ell}}
\newcommand{\cl}{ {m}}
\newcommand{\elips}{ {\mathcal{E}}}
\newcommand{\eme}{ {m}}

\newcommand{\cA}{ {\cal A} }

\newcommand{\cV}{ {\cal V} }

\font\teneufm=eufm10
\font\seveneufm=eufm7
\font\fiveeufm=eufm5
\newfam\eufmfam
\textfont\eufmfam=\teneufm
\scriptfont\eufmfam=\seveneufm
\scriptscriptfont\eufmfam=\fiveeufm

\usepackage{caption}

\newcommand{\wt}{\widetilde}

\renewcommand\subset{\subseteq}

\newcommand{\fc}{\mathfrak{c}}
\newcommand{\T}{\mathbb{T}}
\newcommand{\R}{\mathbb{R}}

\newcommand{\footremember}[2]{%
    \footnote{#2}
    \newcounter{#1}
    \setcounter{#1}{\value{footnote}}%
}
\newcommand{\footrecall}[1]{%
    \footnotemark[\value{#1}]%
    	\newcommand{\pink}[1]{{\color{pink}{#1}}}
	\newcommand{\blue}[1]{{\color{blue}{#1}}}
	\newcommand{\red}[1]{{\color{red}{#1}}}
	\newcommand{\brown}[1]{{\color{brown}{#1}}}
	\newcommand{\black}[1]{\color{black}}
} 
\title{{\bf  }}

\title{{ Birkhoff Conjecture for nearly centrally symmetric domains}}
\author{%
	V. Kaloshin\footremember{alley}{Institute of Science and Technology Austria (ISTA)} \footnote{vadim.kaloshin@gmail.com}%
	\and  C. E. Koudjinan\footrecall{alley}\footnote{koudjinanedmond@gmail.com}%
    \and Ke Zhang\footremember{trailer}{University of Toronto} \footnote{kzhang@math.toronto.edu}%
  }
 \date{\small \today}
\begin{document}

%
\maketitle
\begin{abstract}
\warning{In this paper we prove a perturbative version of a remarkable 
Bialy--Mironov \cite{bialy2022birkhoff} result. They prove non perturbative 
Birkhoff conjecture for centrally-symmetric {convex} domains, 
namely, a centrally-symmetric {convex} domain with integrable 
billiard is ellipse. We combine techniques from Bialy--Mironov \cite{bialy2022birkhoff} 
with a local result by Kaloshin--Sorrentino \cite{kaloshin2018local} and show 
that a domain close enough to a centrally symmetric one with integrable billiard 
is ellipse. {To combine these results we derive a slight extension 
of Bialy--Mironov  \cite{bialy2022birkhoff} by proving that a notion of rational integrability is 
equivalent to the $C^0$-integrability condition used in their paper.}}
\end{abstract}

\section{Introduction}

A mathematical billiard is a system describing the inertial motion of a
point mass inside a domain, with elastic reflections at the boundary (which
is assumed to have infinite mass). This simple model was first proposed by
G. D. Birkhoff as a mathematical playground where “the formal side, usually so
formidable in dynamics, almost completely disappears and only the interesting
qualitative questions need to be considered,” \gwarning{\cite[pp.~155–156]{birkhoff2020periodic}.}

This dynamical system associated to billiards has simple local dynamics, however, 
its study turns out to be really complex and has many important open questions, see
e.g. \gwarning{\cite{gutkin2012billiard}.} In this paper we study integrable billiards and Birkhoff conjecture. 
Let us first recall some properties of the billiard map. We refer e.g. to \gwarning{\cite{tabachnikov2005geometry,siburg2004principle}}, for a more comprehensive introduction to the study of billiards.
\medskip

Let $\O$ be a bounded strictly convex domain in $\real^2$ (for short {\textit{billiard table}}) 
with $C^r$ boundary $\dpr\O$, with $r\ge 3$.\footnote{Observe that if $\O$ is not convex, 
then the billiard map is not continuous; in this article we will be interested only in strictly convex 
domains. Moreover, as pointed out by Halpern \cite{halpern1977strange}, if the boundary is
not at least $C^3$, then the flow might not be complete.} 
The phase space $\mathscr M$ of the billiard map consists of unit vectors $(x, v)$ whose foot 
points $x$ are on $\dpr\O$ and that have inward directions.  The billiard ball map 
$T : \mathscr M\to \mathscr M$ takes $(x, v)$ to $(x',v')$, where $x'$ represents the point where
the trajectory starting at $x$ with velocity $v$ hits the boundary $\dpr\O$ again, and $v'$ is 
the reflected velocity, according to the standard reflection law: angle of incidence is equal to 
the angle of reflection. 
%
\noi
Assume that the boundary $\dpr\O$ is 
parametrized by arc--length $s$ and let $\g : \torus_{|\dpr\O|}  \to \real^2$ denote such 
a parametrization, where $\torus_{|\dpr\O|}\coloneqq \real/ {\mathbb Z \,\cdot}|\dpr\O|$ and $|\dpr\O|$ 
is the length of $\dpr\O$. 
Let $\th$ be the angle between $v$ and the positive tangent to $\dpr\O$ at $s$. Hence, 
$\mathscr M$ can be identified with the annulus 
$\mathbb{A}_\O \coloneqq \torus_{|\dpr\O|} \times (0, \pi)$ and the billiard map $T$ 
can be described as (see \fig~\ref{figgBM})
$$
T\colon (s,\th)\longmapsto (s',\th'),\qquad 
$$
where $(s,\th), (s',\th')\in \mathbb{A}_\O.$
\begin{center}
\includegraphics[scale=1.2]{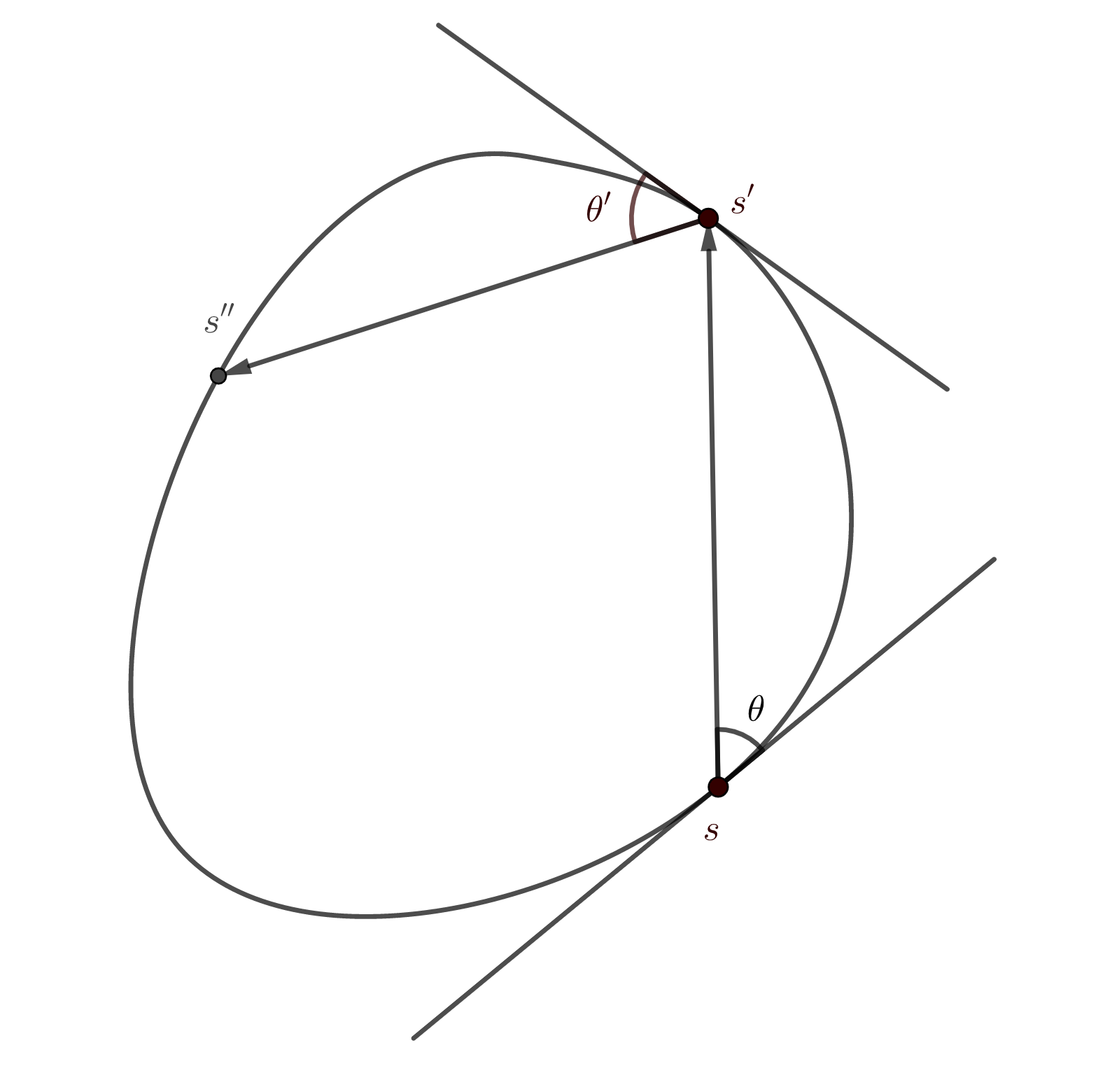} 
  \captionof{figure}{Billiard map: $T(s,\th)= (s',\th')$.} \label{figgBM} 
\end{center}
\ \\
\rwarning{
		A billiard trajectory is called $(p, q)$--periodic if the trajectory is $q$--periodic and the orbit winds around the boundary $p$ times within one period. Equivalently, after lifing the trajectory $\{(s_n, \theta_n)\}_{n \in \mathbb{Z}}$ to the universal cover $\R \times (0, \pi)$, we have $s_q = s_0 + p|\partial\Omega|$. For such orbits, the rotation number is given by $p/q$.  
Furthermore, there exist at least two periodic orbits of rotation number $p/q$,\footnote{In the present paper, rationals are considered in reduced form.} for any rational $p/q$ by a Theorem by Birkhoff \cite{birkhoff1913proof}. 
}

The central objects of this paper are the notions of {\it caustics} and {integrability}.
To state a version of Birkhoff conjecture, we introduce some notions.

\dfn{RatInt} {\bf (i)}  A curve $\mathscr C\subset \O$ is called a caustic for the billiard on the table $\O$ 
if any billiard orbit having one segment tangent to $\mathscr C$ has all its segments tangent to it.

\noi
A caustic is called $(p,q)$--rationally integrable ($p,q\in\natural$ with $p\le q/2$) if all its tangential orbits 
are $(p,q)$--periodic; whenever $p=1$, we shall simply call such a caustic $q$--rationally integrable.

\noi
{\bf (ii)} A billiard table $\O$ is called $q_0$--rationally  integrable (\resp \rwarning{weakly} integrable) if it admits a $(p,q)$--rationally (\resp \rwarning{a $q$--rationally})
integrable caustic for each $0<p/q\le 1/q_0$ (\resp \rwarning{$q\ge q_0$}).
\edfn


We start with a classical result by M. Bialy \cite{bialy1993convex} who proved the following theorem 
concerning global integrability: if the phase space of the billiard ball map is globally foliated 
by continuous invariant curves that are not null--homotopic, then it is a circular billiard.


A strong version of Birkhoff conjecture states that for any $q_0>1$, a $q_0$--\rwarning{rationally weakly}
integrable billiard is a billiard in an ellipse. For domains close to ellipses this conjectures 
was settled for $q_0=3$ in \cite{avila2016integrable,kaloshin2018local} and for 
$q_0=4,5$ in \cite{huang2018nearly}. For $q_0>3$ in a recent work of Koval 
\cite{koval2021domains}. Recently Bialy-Mironov \cite{bialy2017angular} 
proved a remarkable global result stating that any centrally symmetric domain 
that is $C^0-$integrable in the phase space between the boundary and a 4-caustic is an ellipse \cite{bialy2022birkhoff}. 
\medskip 

Our main result is a perturbative version of this result and is a combination 
of \cite{kaloshin2018local} and \cite{bialy2022birkhoff}: {\it a $4$-rationally integrable domain with a $3$-caustic sufficiently close to any centrally symmetric domain is an ellipse.}

\begin{figure}
  \centering
  \includegraphics[width=0.65\linewidth]{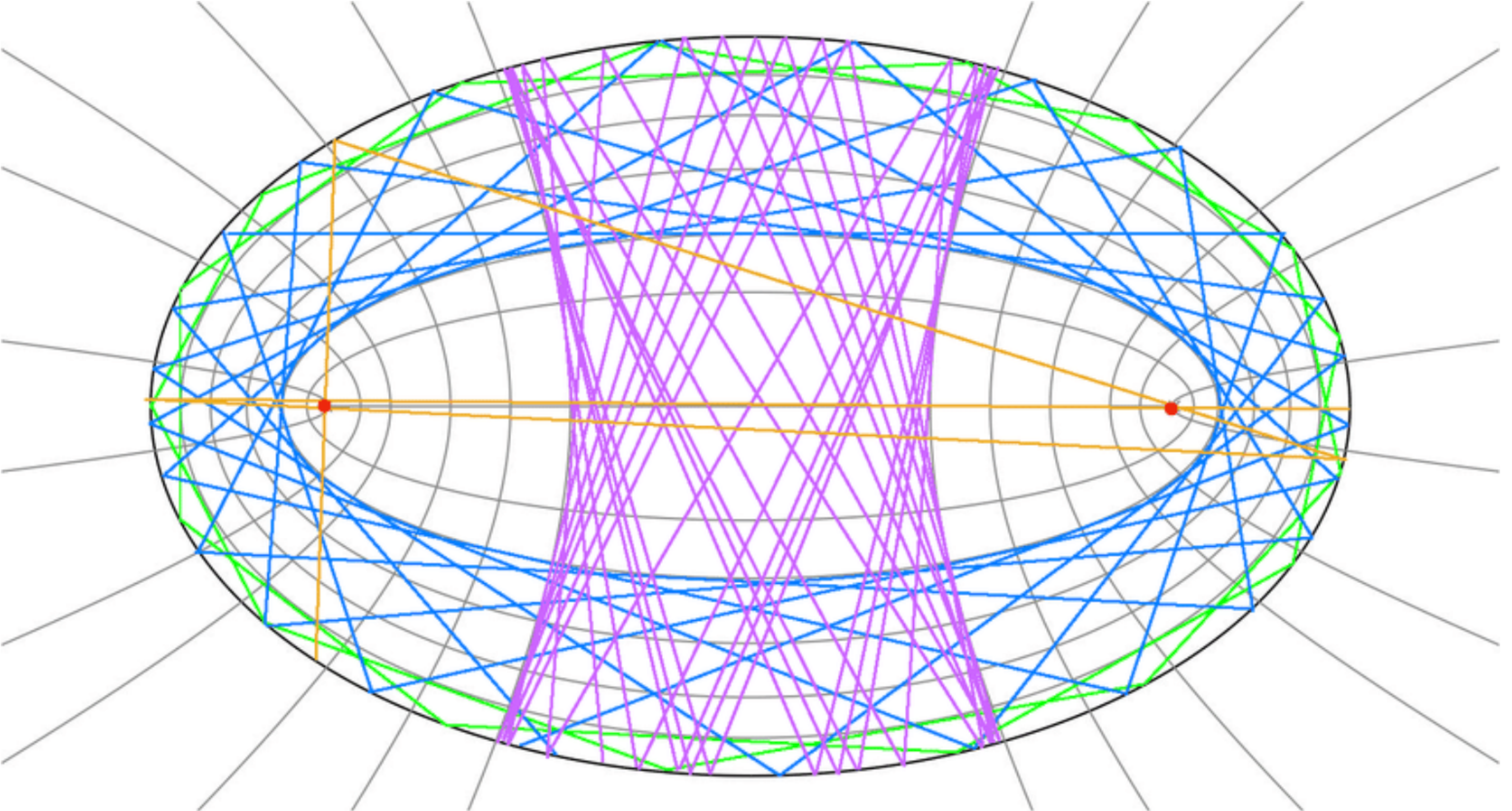}
  \caption{Billiard in an ellipse}
\end{figure}

\subsection{Main result (a technical formulation)}
Let $\Omega$ be a strictly convex subset of $\real^2$ containing the origin $O$, with boundary $\partial \Omega$. It is more convenient to use the support function to represent the boundary, defined by
\[
  h_\Omega(\psi) = \sup \{ x \cos \psi + y \sin \psi: \quad
  (x, y) \in \Omega\}, \qquad \rwarning{\psi\in [0,2\pi]}.
\]
(see \fig~\ref{figgsptfunc}). 
Denoting by $\g_\O(\psi)=(x_\O(\psi),y_\O(\psi))$ the Cartesean coordinates of the point on $\dpr\O$ corresponding to $(\psi,h_\O(\psi))$, we have\footnote{We refer the reader to \cite{resnikoff2015curves} for more details.}
\beq{eqcartsup}
\left\{
\begin{aligned}
&x_\O(\psi)=h_\O(\psi)\cos \psi- h_\O'(\psi)\sin \psi\\
&y_\O(\psi)=h_\O(\psi)\sin \psi+ h_\O'(\psi)\cos \psi,
\end{aligned}
\right.
\eeq
where $h_\O'$ denotes the derivative of $h_\O$. In particular, if the support functions of two boundaries are $C^\cl$-close, \rwarning{then their corresponding parametrization \equ{eqcartsup}  by $\psi$ are $C^{\cl - 1}$-close. 
\newline
Moreover, let $h_\O\in C^\cl$. Then, denoting by $s(\psi)$ the arc--length parameter corresponding to $\psi$, we have
\beq{eqspsi4}
s'(\psi)= \sqrt{x_\O'(\psi)^2+y_\O'(\psi)^2}\eqby{eqcartsup}h_\O(\psi)+h_\O''(\psi)= 1/k_\O(\psi)>0,
\eeq
where $k_\O(\psi)$ is the curvature of $\dpr\O$ \wrt the parameter $\psi$ and for the last equality we refer the reader to \cite{resnikoff2015curves}. The change of parametrization $\psi\mapsto s(\psi)$ is a $C^{\cl-1}$--diffeomorphism.
}

\medskip \noi
\textbf{Example.} The support function of the ellipse $\mathcal{E}_{a,b}\coloneqq \{(x,y)\in\rd:\ x^2/a^2+y^2/b^2=1\}$, $a\ge b>0$, is given by
\beq{sptellipso}
h_{\mathcal{E}_{a,b}}(\psi)=\sqrt{a^2 \cos ^2\psi+b^2 \sin ^2 \psi}=a \sqrt{1-e^2 \sin ^2 \psi}, \qquad e\coloneqq \sqrt{1-\left({b}/{a}\right)^2}\,.
\eeq
\medskip

\begin{center}
\includegraphics[scale=1.5]{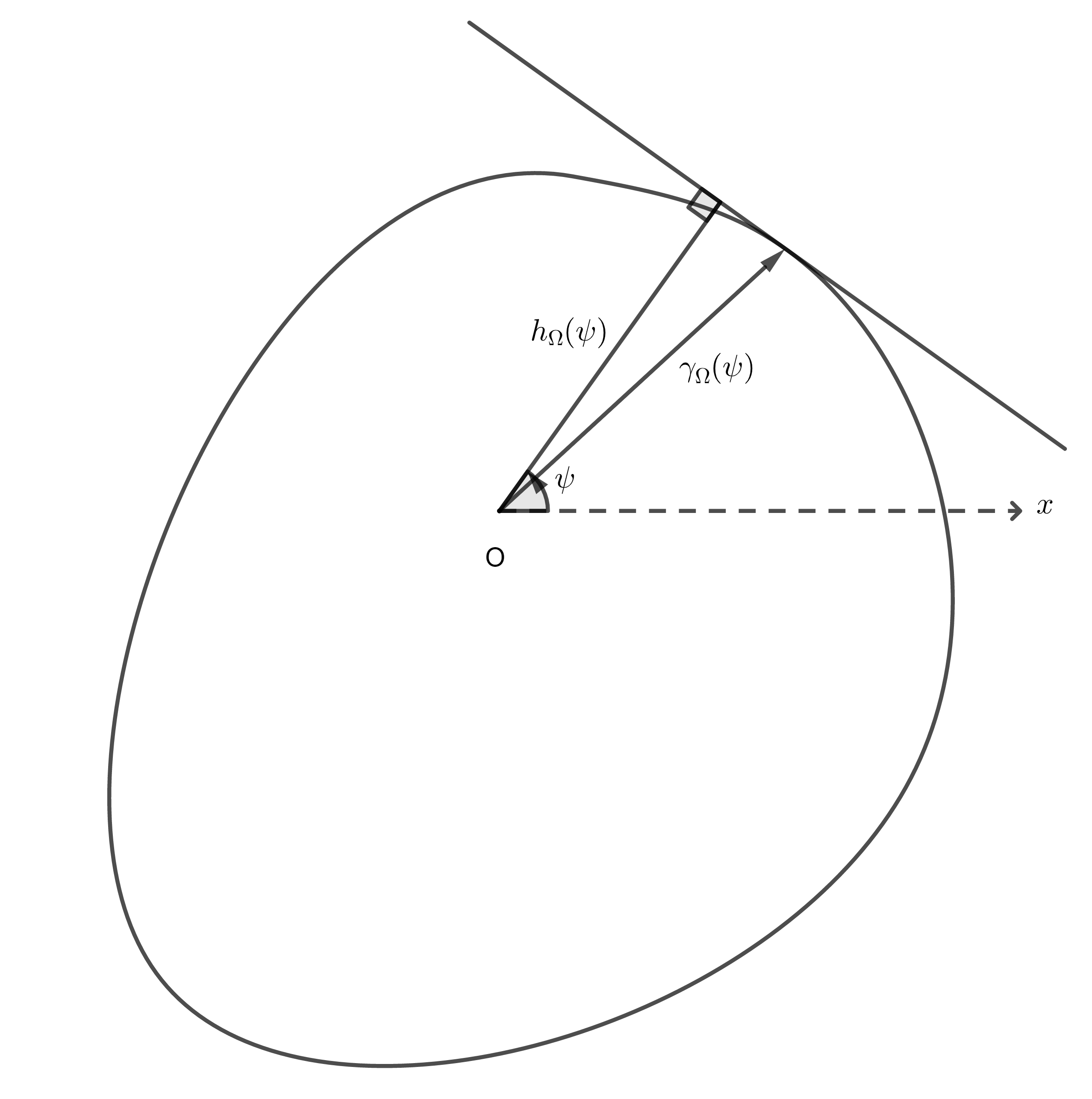} 
\captionof{figure}{Support function of $\Omega$.} \label{figgsptfunc}
\end{center}

{
		\thm{Teo110}
		Let $\O_0$ be a centrally symmetric, strictly convex domain with support function $h_0 \in C^{\cl}(\real/(2\pi \integer))$, with $\cl \ge 40$. Assume that the curvature satisfies
		\[
				k_{\Omega_0}(\psi) > \kappa_0 > 0,
		\]
		and $\|h_0\|_{C^m} < M$. Then there exists $\vae_0 > 0$ depending only on $\kappa_0$, $|\partial \Omega_0|$, and $M$, such that the following holds.

		Assume that the domain $\O_\vae$ with support function $h_\vae$ satisfies
		\[
				\|h_\vae\|_{C^m} < 2M, \quad \|h_\vae - h\|_{C^5} < \vae_0,	
		\]
		and $\Omega_\vae$ admits a $3$--rationally integrable caustic, and satisfies either of the following integrability conditions:
		\begin{enumerate}[(1)]
				\item The phase space between the $4$--rationally integrable caustic and the boundary is foliated by continuous invariant graphs over the $\torus$ component. 
				\item The billiard is $4$-rationally integrable. 
		\end{enumerate}
		Then $\Omega_\vae$ is an ellipse. 
		\ethm
}

\rem{constva0}
\begin{enumerate}[(i)]
		\item	Condition (1) in Theorem \ref{Teo110}, called $C^0$-integrability, is identical to the one in \cite{bialy2022birkhoff}. We will show that condition (2) is equivalent to condition (1), following arguments of \cite{AABZ15}.
		\item There are two different	notions of ``rational integrability'' in the literature. In \cite{kaloshin2018local}, \rwarning{it coreesponds to our ``rational weak integrability''. Our definition of ``rational integrability'' agrees with the one in \cite{koval2021domains} and is strictly stronger (since we also require $(p, q)$-caustics to exist).}

				If a billiard is $4$-rationally integrable, and in addition admits a $3$-rationally integrable caustic, then \rwarning{the billiard is $3$-rationally weakly integrable}, whence verifies the condition in \cite{kaloshin2018local}.
\end{enumerate}
\erem


The reminder of the paper is organized as follows. In $\S$ \ref{sec01}, after recalling the classical ($\S$ \ref{sec011}) and the non--standard ($\S$ \ref{sec012}) symplectic coordinates systems of the billiard map, we provide quantitative estimates on the change between the corresponding parameters ($\S$ \ref{sec013}). In $\S$ \ref{sec4perObt}, we provide some properties of the $4$--periodic orbits in a billiard table which is close to a centrally symmetric one (Theorem~\ref{thm4.1BialMir}). In $\S$ \ref{secNCSPls3RI}, we prove that a nearly--centrally--symmetric and $4$--rationally integrable billiard table is necessarily close to some ellipse $\elips_0$. In $\S$ \ref{ElipPolCoo}, we prove the closeness of the billiard table to the ellipse $\elips_0$ in the elliptic polar coordinates associated to $\elips_0$ (Lemma~\ref{lem:elliptic-cor}). 
In $\S$ \ref{ProfMainteo}, we complete the proof of the main Theorem~\ref{Teo110}. We collect in $\S$ Appendix~\ref{TecfAxt} some technical facts. In $\S$ Appendix~\ref{interprEz}, we prove an interpolation--type result. 
\rwarning{
		In $\S$ Appendix \ref{sec:d0}, we prove the uniqueness of the $(p, q)$-loop orbit, which is a more general version of $q$-loop orbits needed for our proof. See Definition \ref{def:q-loop}.
}
In $\S$ Appendix~\ref{EquivInteg}, we prove the equivalence between the rational and the $C^0$--integrability.

\section{Generating functions and some auxiliary facts \label{sec01}}

\subsection{The traditional billiard generating function \label{sec011}}

Denote the Euclidean distance between two points of $\dpr\O$ by
$$
l(s,s')\coloneqq |\g(s)-\g(s')|.
$$
Then, one checks easily
\beq{gen005}
\dpr_s l(s,s')=-\cos \th,\qquad \dpr_{s'} l(s,s')=\cos \th',
\eeq
where $T(s,\theta)=(s',\theta')$. In particular, if we lift everything to the universal cover and 
introduce the new coordinates $(x,y)=(s,\cos \th)\in\real\times [-1,1]$, then the billiard map 
is a twist map with $l$ as generating function, and it preserves the area form $dx\wedge dy$.\footnote{See \cite{tabachnikov2005geometry, siburg2004principle} for details.}

For the main text of this paper, we will only use the Bialy--Mironov generating function introduced in the next section. However, we do use the traditional generating function in the Appendix.

\subsection{Bialy--Mironov generating function \label{sec012}}
The following non--standard generating function for the billiard map has been discovered by Bialy and Mironov\cite{bialy2017angular} (see \fig~\ref{figgBMGef})

$$S(\f,\f')={2}\,h\left(\frac{\f+\f'}{2}\right)\sin\left(\frac{\f'-\f}{2}\right).$$
\begin{center}
\includegraphics[scale=1]{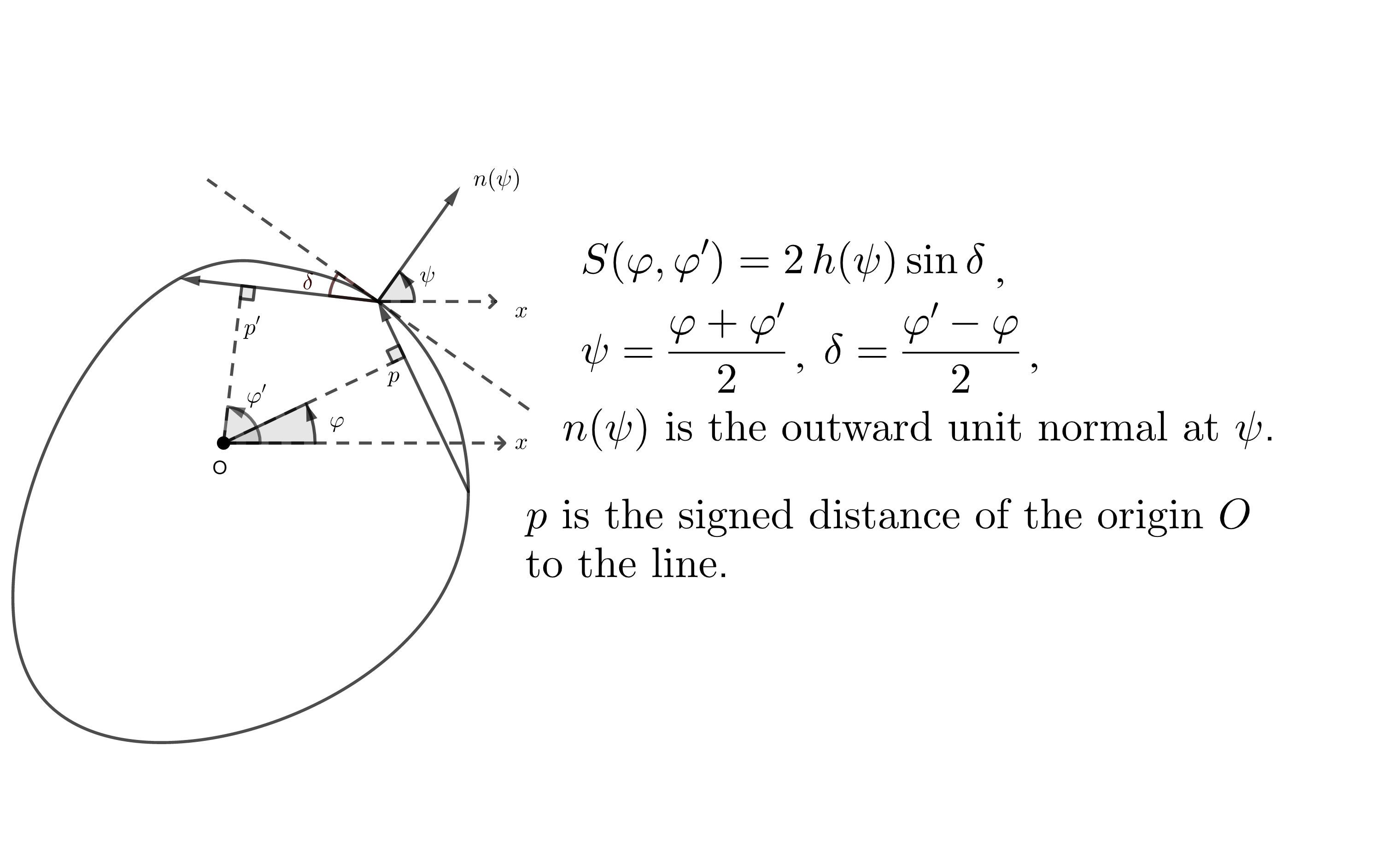} 
  \captionof{figure}{Non standard Bialy--Mironov generating function.} \label{figgBMGef} 
\end{center}

\noi
In the $(\f,p)$--coordinates, the billiard map $T$ reads: $T(\f,p)=(\f',p')$ \sss
\begin{equation}  \label{bm-gen}
\left\{
\begin{aligned}
p&= -S_1(\f,\f') = h(\psi) \cos \delta - h'(\psi) \sin \delta\\
p'&= S_2(\f,\f') = h(\psi) \cos \delta + h'(\psi) \sin \delta
\end{aligned}
\right.
\end{equation}
where $\psi = \frac12 (\varphi + \varphi')$, $\delta = \frac12 (\varphi' - \varphi)$, and $S_1\coloneqq \dpr_\f S$, $S_2\coloneqq \dpr_{\f'} S$, (for later use) $S_{12}\coloneqq \dpr_{\f \f'}^2 S$, \etc.

\noi
Denote by $\g_\vae$ the $\psi$--parametrization of $\dpr\O$ given by \equ{eqcartsup}, by $S_\vae$ its Bialy--Mironov generating function and by $T_\vae$ the billiard map in the coordinate $(\f,p)$.

\noi

{
		\subsection{Some estimates using the support function parametrization \label{sec013}}
}

Given $\vae_1 > 0$, define the space
\begin{equation}  \label{eq:V}
		\cV = \{h \in C^m(\T): \quad \|h\|_{C^m} < 2M, \quad \|h - h_0\|_{C^5} < \vae_1\}.  
\end{equation}
We assuem that $\vae_1$ is small enough so that the curvature function associated to $h \in \cV$ satsifies
\[
   k_\Omega(\psi) > \kappa_0/2, \quad \forall \psi \in \T. 
\]
Moreover, $\vae_1$ will be chosen so that Proposition \ref{prop:4-loop} applies, which depends only on $\kappa_0$, $|\partial \Omega_0|$, $M$. 

Throughout the proof we assume $h_\vae \in \cV$ and set
\[
		\vae = \|h_\vae - h_0\|_{C^5}.   
\]
We shall write $f = O_n(\vae)$ if $\|f\|_{C^n} \le C_1 \vae$ for a constant $C_1$ depending only on $\kappa_0$, $|\partial \Omega_0|$, $M$. In particular, $f = O_n(1)$ means $f$ has bounded $C^n$ norm. 

\noi
Denote by $s_\vae$ the arc--length parameter of $\dpr\O$. Then,
\lem{vapsis}
The change of parametrization $\psi\mapsto s_\vae(\psi)$ is a $C^{\cl-1}$--diffeomorphism with  inverse $s\mapsto \psi_\vae(s)$, and we have
\beq{psivaesvae0}
\psi_\vae = O_{m-1}(1), \quad
s_\vae = O_{m-1}(1),\quad  
\wt\psi_\vae=O_{4}(\vae),\quad \tilde s_\vae=O_{4}(\vae),
\eeq
where $\wt\psi_\vae\coloneqq \psi_\vae-\psi_0$ and $\tilde s_\vae\coloneqq s_\vae-s_0$.
\elem
\proof Indeed, \equ{eqspsi4} implies $s_\vae$ is a $C^{\cl-1}$--diffeomorphism and $s_\vae'=h_\vae+h_\vae''= h_0+h_0''+O_{3}(\vae)=s_0'+O_{3}(\vae)$. Thus,  $\wt s_\vae=O_{4}(\vae)$. Therefore, by the implicit function theorem, 
its inverse $\psi_\vae$ satisfies $ \tilde \psi_\vae=O_{4}(\vae)$.
\qed

\section{Proof of Theorem~\ref{Teo110}}

\subsection{Properties of $4$--periodic orbits in nearly--centrally--symmetric billiard tables \label{sec4perObt}}
In this section, we prove that the properties of the one parameter family of $4$--gons ``persists'' up to small corrections on tables which are close to a centrally--symmetric one and admits a $4$--rationally integrable caustic. 

{ 
\dfn{def:q-loop}
Given $q \ge 2$, an orbit \rwarning{ segment } of the billiard is called a \emph{q-loop orbit} if the \rwarning{segment} bounces $q$ times back to its starting point, winding around the table exactly once. 
\edfn

\rem{rm-q-loop}
A q-loop orbit is not necessarily a periodic orbit, as the final angle of incidence is not required to equal to the inital angle of reflection. 
\erem

{
		For a billiard $\Omega$ with parametrized boundary $\gamma$, the $q$-loop orbit starting at any $\gamma(\psi)$ always exists, by taking the maximum perimeter $q$-gon that starts and ends at $\gamma(\psi)$. It is unique if the billiard is nearly circular, see \cite{hezari2022one}. In Proposition \ref{prop:4-loop}, we show that the $4$-loop orbit exists for both the billiards $\Omega_0$ and $\Omega_\vae$. 

		\begin{proposition}\label{prop:4-loop}
				There exists $\vae_2 > 0$ and $C > 0$ depending only on $\kappa_0$ and $\|\gamma_0\|_{C^4}$, such that the following hold.

				Consider the space of closed curves $\gamma: \T \to \R^2$
				\[
						\cV_\gamma = \{\gamma:	\|\gamma - \gamma_0\|_{C^2} < \vae_2, \quad \|\gamma\|_{C^4} < C\},	
				\]
				then each $\gamma$ is the boundary of a convex billiard (since $\vae_2$ is small enough). Suppose there exists at least one $\gamma_\vae \in \cV_\gamma$ such that $\gamma_\vae$ admits a $4$-caustic, then for all $\gamma \in \cV_\gamma$, the corresponding billiard admits unique $4$-loop orbit starting at every point on the boundary.
		\end{proposition}

		Proposition \ref{prop:4-loop} follows from a more general result for $(p, q)$-loop orbit, for both twist map and billiards. The proof is presented in Appendix \ref{sec:d0}.
}


\dfn{def4loopang}
Assume the billiard $\Omega$ admits a unique $4$-loop orbit starting at any point on the boundary. We define $d_{\O}(\psi)\in (0,\pi)$ to be the unique initial angle for the unique 4-loop orbit of the billiard in $\Omega$ starting at the point $\gamma(\psi)$.(see \fig~\ref{figg4LopFun}). 
\edfn
}

\begin{center}
\includegraphics[scale=1.35]{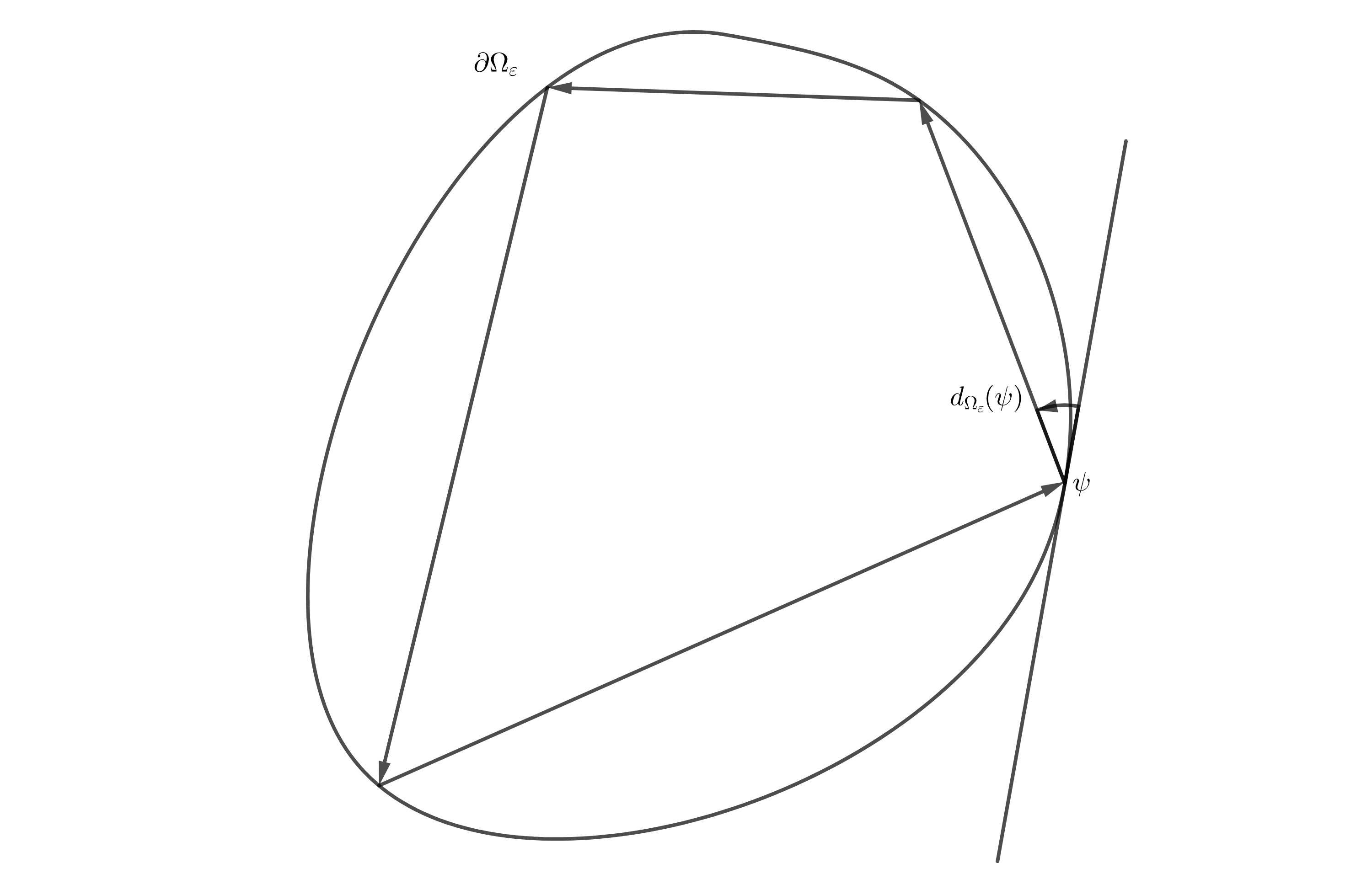} 
  \captionof{figure}{The $4$--loop angle function $d_{\O_\vae}$ in $\dpr\O_\vae$.} \label{figg4LopFun} 
\end{center}

We can choose $\vae_1$ in \eqref{eq:V} small enough so that $h_\vae \in \cV$ implies $\gamma_\vae \in \cV_\gamma$. Therefore the $4$-loop angle functions $d_\vae$ and $d_0$ are both defined.

{ 
		The following lemma is derived from the proof of Proposition \ref{prop:4-loop}, whose proof is in Appendix \ref{sec:d0}.
		\begin{lemma}\label{lem:d-estimates}
				\begin{enumerate}[(i)]
						\item $d_\vae$ is $C^{m-1}$ with $d_\vae = O_{m-1}(1)$, and $\tilde{d}_\vae = d_\vae - d_0 = O_3(\vae)$.
						\item There exists $0 < \underline{d} < \overline{d} < \pi$ such that $\underline{d} < d_0, d_\vae < \overline{d}$.
				\end{enumerate}
		\end{lemma}
}

We have the following.

\thmtwo{thm4.1BialMir}{Perturbative Theorem~4.1 \cite{bialy2022birkhoff}}\ \\
(i) $d_\vae(\psi+\pi)=d_\vae(\psi)+O_{3}(\vae)$ 
and each $4$--periodic billiard trajectory is $O_{2}(\vae)$--close to a parallelogram.\\
(ii) The tangent line to $\O$ at the vertices of any $4$--periodic billiard trajectory form a quadrilateral which is $O_{2}(\vae)$--close to a rectangle.\\
\rwarning{(iii) $d_\vae(\psi+\frac{\pi}{2})=\frac{\pi}{2}- d_\vae(\psi)+O_{2}(\vae)$.}\\
(iv) There exists $R_0>0$ such that
\begin{align}
& h_\vae^2(\psi)+h_\vae^2(\psi+\frac{\pi}{2})=R_0^2+O_{3}(\vae)\eqqcolon R_\vae^2\,,\label{reldh2}\\
& { h_\vae(\psi)=R_0\sin d_\vae (\psi)+O_{2}(\vae),\qquad h_\vae(\psi+\frac{\pi}{2})=R_0\cos d_\vae (\psi)+O_{2}(\vae)}\label{reldh1}\,.
\end{align}
\ethm

\proof
(i) Observe that, by Lemma~\ref{lem1}, the central symmetry of $\O_0$ implies $d_0(\psi+\pi)=d_0(\psi)$. Then, by Lemma~\ref{lem:d-estimates}, we have
$$
d_\vae(\psi+\pi)=d_0(\psi+\pi)+\tilde d_\vae(\psi+\pi)= d_0(\psi)+\tilde d_\vae(\psi+\pi)= 
$$
$$
\qquad \qquad  \qquad d_\vae(\psi)+\tilde d_\vae(\psi+\pi)-\tilde d_\vae(\psi)=d_\vae(\psi)+O_{3}(\vae).
$$
Next, we show the last part of (i): writing
\beq{ceNTsYm0}
\Psi_\vae^2(\psi)=\psi+\pi+f(\psi), 
\eeq
it is enough to show
\beq{ceNTsYm}
 \qquad f=O_{2}(\vae).
\eeq
Indeed,  by Lemma~\ref{lem19}
\begin{align*}
\Psi_\vae^3(\psi)&=\Psi_\vae(\psi+\pi+f(\psi))\\
	&=\Psi_\vae(\psi+\pi)+\int_0^1\Psi_\vae'(\psi+\pi+t f(\psi))dt\cdot f(\psi)\\
	&= \Psi_\vae(\psi)+\pi+\wt{\Psi}_\vae(\psi+\pi)-\wt{\Psi}_\vae(\psi)+\int_0^1\Psi_\vae'(\psi+\pi+t f(\psi))dt\cdot f(\psi)\\
	&\eqqcolon \Psi_\vae(\psi)+\pi+f_1(\psi),
\end{align*}
and
\begin{align*}
\psi+2\pi&=\Psi_\vae^4(\psi)\\
	&=\Psi_\vae(\Psi_\vae(\psi)+\pi+f_1(\psi))\\
	&=\Psi_\vae(\Psi_\vae(\psi)+\pi)+\int_0^1\Psi_\vae'(\Psi_\vae(\psi)+\pi+t f_1(\psi))dt\cdot f_1(\psi)\\
	&= \Psi_\vae^2(\psi)+\pi+\wt{\Psi}_\vae(\Psi_\vae(\psi)+\pi)-\wt{\Psi}_\vae(\Psi_\vae(\psi))+\int_0^1\Psi_\vae'(\Psi_\vae(\psi)+\pi+t f_1(\psi))dt\cdot f_1(\psi)\\
	&= \psi+\pi+f(\psi)+\pi+\wt{\Psi}_\vae(\Psi_\vae(\psi)+\pi)-\wt{\Psi}_\vae(\Psi_\vae(\psi))+\int_0^1\Psi_\vae'(\Psi_\vae(\psi)+\pi+t f_1(\psi))dt\cdot f_1(\psi)\\
	&\eqqcolon \psi+2\pi+f_2(\psi),
\end{align*}
\ie
$$
f_2(\psi)\equiv 0,
$$
where
\begin{align}
f_1(\psi)&\coloneqq \wt{\Psi}_\vae(\psi+\pi)-\wt{\Psi}_\vae(\psi)+\int_0^1\Psi_\vae'(\psi+\pi+t f(\psi))dt\cdot f(\psi),\label{formf1}\\
f_2(\psi)&\coloneqq f(\psi)+\wt{\Psi}_\vae(\Psi_\vae(\psi)+\pi)-\wt{\Psi}_\vae(\Psi_\vae(\psi))+\int_0^1\Psi_\vae'(\Psi_\vae(\psi)+\pi+t f_1(\psi))dt\cdot f_1(\psi)\nonumber\\
	&= \bigg(1+ \Psi_\vae'(\Psi_\vae(\psi)+\pi+t_1 f_1(\psi))\cdot \Psi_\vae'(\psi+\pi+t_2 f(\psi))\bigg)f(\psi)+\nonumber\\
	& + \wt{\Psi}_\vae(\Psi_\vae(\psi)+\pi)-\wt{\Psi}_\vae(\Psi_\vae(\psi))+(\wt{\Psi}_\vae(\psi+\pi)-\wt{\Psi}_\vae(\psi))\Psi_\vae'(\Psi_\vae(\psi)+\pi+t_3 f_1(\psi)), \label{formf2}
\end{align}
for some $t_1,t_2,t_3\in (0,1)$.


\noi
Hence, as  
$0\equiv f_2(\psi)$, 
it follows
\begin{align*}
&-\bigg(1+ \Psi_\vae'(\Psi_\vae(\psi)+\pi+t_1 f_1(\psi))\cdot \Psi_\vae'(\psi+\pi+t_2 f(\psi))\bigg)f(\psi)= 
\\
& \qquad \wt{\Psi}_\vae(\Psi_\vae(\psi)+\pi)-\wt{\Psi}_\vae(\Psi_\vae(\psi))+\\
&\hspace{2cm}	\qquad+(\wt{\Psi}_\vae(\psi+\pi)-\wt{\Psi}_\vae(\psi))\cdot \Psi_\vae'(\Psi_\vae(\psi)+\pi+t_3 f_1(\psi))\,.
\end{align*}
By Lemma~\ref{lem19}, $\wt{\Psi}_\vae=O_{3}(\vae)$. 
Now, observe that $\Psi_\vae$ is strictly increasing by the strict convexity of the caustic $\mathscr C$. Then, $\Psi_\vae'>0$ and  
{
\beq{PSiWtdeRv}
1+ \Psi_\vae'(\psi_1)\cdot \Psi_\vae'(\psi_2)> 1,\qquad \forall\, \psi_1,\psi_2\in \torus,
\eeq
%
}
Thus, the Fa\`a di Bruno's formula yields \equ{ceNTsYm} by induction.

\noi
(ii) Indeed,
\beq{eqrEcT}
d_\vae(\Psi_\vae^2(\psi))\eqby{ceNTsYm0} d_\vae(\psi+\pi+f_1(\psi))\overset{\equ{formf1},\equ{ceNTsYm},\equ{estPsiVaE}}= d_\vae(\psi+\pi)+ O_{2}(\vae)\overset{(i)}=d_\vae(\psi)+ O_{2}(\vae),
\eeq
so that, as in \cite{bialy2022birkhoff},
\begin{align*}
2\pi&=2\left(d_\vae(\psi)+d_\vae(\Psi_\vae(\psi))+d_\vae(\Psi_\vae^2(\psi))+d_\vae(\Psi_\vae^3(\psi))\right)\\
	&\eqby{eqrEcT} 4\left(d_\vae(\psi)+d_\vae(\Psi_\vae(\psi))\right)+O_{2}(\vae)
\end{align*}
\ie 
\beq{RectY0}
d_\vae(\psi)+d_\vae(\Psi_\vae(\psi))= \frac{\pi}{2}+O_{2}(\vae),
\eeq
proving (ii).

\noi
(iii)
Recalling that $\psi$ is the angle formed by the outer unit normal of $\dpr\O_\vae$ at $\psi$ with the $x$--axis, \equ{RectY0} then implies
\beq{RectY06}
\Psi_\vae(\psi)= \psi+\frac{\pi}{2}+O_{2}(\vae),
\eeq
which, together with \equ{RectY0} imply the last assertion in (iii). 

\noi
{Next, we prove $0<\uvl d\le d_\vae(\psi)\leq \ovl d<\pi/2$.} 
Indeed, $d_\vae$ being continuous on the compact $\torus$ then, it admits a minimum and a maximum \ie there are $\psi_{min},\psi_{max}\in\torus$ \st $d_\vae(\psi_{min})\le d_\vae(\psi)\le d_\vae(\psi_{max})$, forall $\psi\in\torus$. Obviously, $d_\vae(\psi)>0$ forall $\psi\in\torus$. Thus, $d_\vae(\psi_{min})>0$, and $0<d_\vae(\psi_{max}+\frac{\pi}{2})=\frac{\pi}{2}- d_\vae(\psi_{max})+O_{3}(\vae)$ which implies that $\frac{\pi}{2}- d_\vae(\psi_{max})>0$, concluding the proof of (iii). 
\newline

\noi
(iv) First of all, by Lemma~\ref{lem1}, 
\beq{hepUn8}
h_\vae(\psi+\pi)=h_0(\psi+\pi)+\tilde h_\vae(\psi+\pi)=h_0(\psi)+\tilde h_\vae(\psi+\pi)=h_\vae(\psi)+\tilde h_\vae(\psi+\pi)-\tilde  h_\vae(\psi).
\eeq
Then, just as in \cite{bialy2022birkhoff}, we obtain on one hand
\begin{align*}
h_\vae(\psi)\cos d_\vae(\psi)&+h_\vae'(\psi)\sin d_\vae(\psi) \\
                             &\eqby{bm-gen} h_\vae(\Psi_\vae(\psi))\cos d_\vae(\Psi_\vae(\psi))-h_\vae'(\Psi_\vae(\psi))\sin d_\vae(\Psi_\vae(\psi))\\
&\eqby{RectY06} h_\vae(\psi+\frac{\pi}{2})\cos d_\vae(\psi+\frac{\pi}{2})-h_\vae'(\psi+\frac{\pi}{2})\sin d_\vae(\psi+\frac{\pi}{2})+O_{2}(\vae)\,,
\end{align*}
and, on the other hand
\begin{align*}
& h_\vae(\psi+\frac{\pi}{2})\cos d_\vae(\psi+\frac{\pi}{2})-h_\vae'(\psi+\frac{\pi}{2})\sin d_\vae(\psi+\frac{\pi}{2})+O_{2}(\vae)\eqby{RectY06}\\
&\hspace{3cm}\eqby{RectY06} h_\vae(\Psi_\vae(\psi))\cos d_\vae(\Psi_\vae(\psi))-h_\vae'(\Psi_\vae(\psi))\sin d_\vae(\Psi_\vae(\psi))\\
&\hspace{3cm}\eqby{bm-gen} h_\vae(\Psi_\vae^2(\psi))\cos d_\vae(\Psi_\vae^2(\psi))+h_\vae'(\Psi_\vae^2(\psi))\sin d_\vae(\Psi_\vae^2(\psi))\\
&\hspace{3cm}\overset{\equ{ceNTsYm0},\equ{ceNTsYm}}= h_\vae(\psi+\pi)\cos d_\vae(\psi+\pi)-h_\vae'(\psi+\pi)\sin d_\vae(\psi+\pi)+O_{2}(\vae)\\
&\hspace{3cm} \overset{\equ{hepUn8}}=h_\vae(\psi)\cos d_\vae(\psi+\pi)-h_\vae'(\psi)\sin d_\vae(\psi+\pi)+O_{2}(\vae)\\
&\hspace{3cm} \overset{(i)}=h_\vae(\psi)\cos d_\vae(\psi)-h_\vae'(\psi)\sin d_\vae(\psi)+O_{2}(\vae)\,,
\end{align*}
and, summing and subtracting up yields
$$
\left\{
\begin{aligned}
&h_\vae(\psi+\frac{\pi}{2})\cos d_\vae(\psi+\frac{\pi}{2})=h_\vae(\psi)\cos d_\vae(\psi)+O_{2}(\vae)\,,\\
&h_\vae'(\psi+\frac{\pi}{2})\sin d_\vae(\psi+\frac{\pi}{2})=-h_\vae'(\psi)\sin d_\vae(\psi)+O_{2}(\vae)\,,
\end{aligned}
\right.
$$
which, combined with $d_\vae(\psi+\frac{\pi}{2})=\frac{\pi}{2}- d_\vae(\psi)+O_{2}(\vae)$ (\textit{cf.} (iii)) yields

\beq{uLtG}
\left\{
\begin{aligned}
&h_\vae(\psi+\frac{\pi}{2})\sin d_\vae(\psi)=h_\vae(\psi)\cos d_\vae(\psi)+O_{2}(\vae)\,,\\
&h_\vae'(\psi+\frac{\pi}{2})\cos d_\vae(\psi)=-h_\vae'(\psi)\sin d_\vae(\psi)+O_{2}(\vae)\,.
\end{aligned}
\right.
\eeq
{Now, multiplying the two equations in \equ{uLtG} side--wise, we obtain
$$
\left(h_\vae(\psi+\frac{\pi}{2})h_\vae'(\psi+\frac{\pi}{2})+h_\vae(\psi)h_\vae'(\psi)\right)\sin 2d_\vae(\psi)=O_{2}(\vae)\,,
$$
which together with
$$
\sin 2d_\vae(\psi)=\sin 2d_0(\psi)+2\tilde{d}_\vae(\psi)\int_0^1\cos(2d_0(\psi)+t  2\tilde{d}_\vae(\psi))dt\overset{Lemma~\ref{lem:d-estimates}}=\sin 2d_0(\psi)+O_{3}(\vae)\,,
$$
imply
\begin{align*}
\su2 \left(h_\vae^2(\psi)+h_\vae^2(\psi+\frac{\pi}{2})\right)'=h_\vae'(\psi)h_\vae(\psi)+h_\vae'(\psi+\frac{\pi}{2})h_\vae(\psi+\frac{\pi}{2})
= O_{2}(\vae),
\end{align*}
proving \equ{reldh2}, \rwarning{$R_0$ being the constant of integration from $0$ to $\psi$}.\\
\noi
From \equ{reldh2}, it follows 
\beq{bfdvEa}
h_\vae(\psi)=R_\vae \sin \mathbf{d}_\vae(\psi),\qquad h_\vae(\psi+\frac{\pi}{2})=R_\vae \cos \mathbf{d}_\vae(\psi),
\eeq
for some $\mathbf{d}_\vae(\psi)\in [0,\pi/2]$ since $h_\vae>0$. Since $h_\vae>0$, \equ{bfdvEa} yields $\mathbf{d}_\vae(\psi)=-i\log\su{R_\vae}(h_\vae(\psi+\frac{\pi}{2})+i h_\vae(\psi))$ and, hence, $\mathbf{d}_\vae(\psi)$ is $C^{\eme+1}$--smooth. Now, plugging \equ{bfdvEa} in the first equation in  \equ{uLtG}, we obtain $R_\vae \sin (\mathbf{d}_\vae-d_\vae)=O_{2}(\vae)$ and, by definition of $R_\vae$ in \equ{reldh2}, yields $ \sin (\mathbf{d}_\vae-d_\vae)=O_{2}(\vae)$. Then , by continuity, $\mathbf{d}_\vae(\psi)-d_\vae(\psi)\eqby{reldh2}n_\vae\pi+ O_{2}(\vae)$, for some $n_\vae\in\integer$. 
It turns out that $n_\vae=0$ since $d_\vae(\psi),\mathbf{d}_\vae(\psi)\in [0,\pi/2]$. Hence, plugging $\mathbf{d}_\vae(\psi)=d_\vae(\psi)+ O_{2}(\vae)$ in \equ{bfdvEa} and recalling the definition of $R_\vae$ as in \equ{reldh2}, we obtain \equ{reldh1}.
} 
\qed

\subsection{Nearly--centrally--symmetry and $3$--rational integrability imply closeness to an ellipse\label{secNCSPls3RI}}
{
Denote
\gwarning{$$
U_\vae\coloneqq -h_\vae(h_\vae')^2(h_\vae+h_\vae'')\left(\su 2 d_\vae-\su 4\sin 2d_\vae \right)+(h_\vae''h_\vae^2+3h_\vae(h_\vae')^2)(h_\vae+h_\vae'')\left(\su 8 d_\vae-\su {32}\sin 4d_\vae \right).
$$
}
We have the following statement.
\pro{} (see \cite{bialy2022birkhoff})
Suppose for the billiard $\Omega_\vae$, the part of the phase space between the $4$-caustic and the boundary is foliated by $C^0$--rotational invariant curves, then 
\beq{EqIntBmR}
\int_0^{\warning{2\pi}} U_\vae\,d\psi\le 0.
\eeq
\epro
\proof
We refer to \cf \cite[Eq.~(19)]{bialy2022birkhoff}, noting that the proof in \cite[section~5.1]{bialy2022birkhoff}  there up to this formula does not use the central symmetry assumption. 
\qed
The assumption of $C^0$-integrability can be replaced by rational integrability.
\pro{} (See Proposition \ref{prop:rat-int})
The billiard $\Omega$ is $q_0$-rationally integrable if and only if the part of the phase space between the $q_0$-caustic and the boundary is foliated by $C^0$--rotational invariant curves.
\epro
}

We now proceed with \eqref{EqIntBmR}.
\lemtwo{lem5.1BialMir}{Perturbative Lemma~5.1\cite{bialy2022birkhoff}}\ \\
\beq{RedWirtBiaMir9}
\int_0^{2\pi} U_\vae(\psi)d\psi=\frac{\pi R_0^4}{512}\int_0^{2\pi} ((\m_\vae'')^2-4(\m_\vae')^2)d\psi+O(\vae),
\eeq
with $\m_\vae(\psi)=\cos(2d_\vae(\psi))$.
\elem

\proof 
In our setting, \cite[Eq.~(20)]{bialy2022birkhoff} reads
\beq{Eq20BilMir}
\left\{
\begin{aligned}
&h_\vae=R_0\sin d_\vae+O_{2}(\vae)\,,\\
&h_\vae'=R_0\cos d_\vae\, d_\vae'+O_{1}(\vae)\,,\\
&h_\vae''=R_0\cos d_\vae\, d_\vae''-R_0\sin d_\vae (d_\vae')^2+O_{ }(\vae)\,,\\
&d_\vae(\psi+\frac{\pi}{2})=\frac{\pi}{2}- d_\vae(\psi)+O_{2}(\vae)\,.
\end{aligned}
\right.
\eeq
Thus, mimicking the proof of \cite[Lemma~5.1]{bialy2022birkhoff}, we obtain \equ{RedWirtBiaMir9}.
\qed
\newline
\rwarning{Writing $\mu_\vae$, $\mu_0$ in Fourier series as $\mu_\vae\eqqcolon \sum_{j\in\integer} \hat{\m}_{\vae,j}\ex^{ij\psi}$ and $\mu_0 \eqqcolon \sum_{j\in\integer} \hat{\m}_{0,j}\ex^{ij\psi}$, we have:}
\lem{ClsElps} We have $\m_\vae-(\hat{\m}_{0,-2}\ex^{-i2\psi}+\hat{\m}_{0,2}\ex^{i2\psi})=O_1(\sqrt\vae)$.
\elem

\proof Indeed, by \equ{RedWirtBiaMir9} and \equ{EqIntBmR}, we have, for some $C>0$
\beq{EqUtns}
\int_0^{2\pi} ((\m_\vae'')^2-4(\m_\vae')^2)d\psi\le C\,\vae.
\eeq
 Since $\m_0=\cos(2d_0)$ is $\pi$--periodic\footnote{Recall that $d_0$ is $\pi$--periodic.}, we have $\hat{\m}_{0,2k+1}=0$ for all $k\in\integer$ so that
 \rwarning{
\begin{align}
|\hat{\m}_{\vae,\pm 1}|&= |\hat{\m}_{\vae,\pm 1}-\hat{\m}_{0,\pm 1}|=|\int_0^{2\pi} (\m_{\vae}-\m_{0})(\psi)\ex^{\pm i\psi} d\psi|\le 2\pi\|\m_{\vae}-\m_{0}\|_{C^0}=\nonumber\\
	&=4\pi\|\sin(d_\vae-d_0)\sin(d_\vae+d_0)\|_{C^0}\le 4\pi\|\sin(d_\vae-d_0)\|_{C^0} \overset{Lemma~\ref{lem:d-estimates}}=O(\vae).\label{mEpspm1}
\end{align}
}
\noi
Thus, 
\rwarning{
\begin{align*}
\sum_{k\in \integer\setminus\{0,\pm 2\}} k^4|\hat{\m}_{\vae,k}|^2&=(|\hat{\m}_{\vae,-1}|^2+|\hat{\m}_{\vae,1}|^2)+ \sum_{|k|\ge 3} k^4|\hat{\m}_{\vae,k}|^2\\
	&\le (|\hat{\m}_{\vae,-1}|^2+|\hat{\m}_{\vae,1}|^2)+ \sum_{|k|\ge 3} (k^2-4)k^2|\hat{\m}_{\vae,k}|^2\\
	&= 4(|\hat{\m}_{\vae,-1}|^2+|\hat{\m}_{\vae,1}|^2)+ \sum_{k\in \integer} (k^2-4)k^2|\hat{\m}_{\vae,k}|^2\\
	&=4(|\hat{\m}_{\vae,-1}|^2+|\hat{\m}_{\vae,1}|^2)+\su{2\pi}\int_0^{2\pi} ((\m_\vae'')^2-4(\m_\vae')^2)d\psi \\
	&\qquad\quad (\mbox{by Parseval’s identity})\\
	&\overset{\equ{EqUtns},\equ{mEpspm1}}\le C\,\vae\,,
\end{align*}
}
\ie $\m_\vae-(\hat{\m}_{\vae,0}+\hat{\m}_{\vae,-2}\ex^{-i2\psi}+\hat{\m}_{\vae,2}\ex^{i2\psi})\in H^2(\torus)$ and $\|\m_\vae-(\hat{\m}_{\vae,0}+\hat{\m}_{\vae,-2}\ex^{-i2\psi}+\hat{\m}_{\vae,2}\ex^{i2\psi})\|_{H^2}\le C\,\sqrt{\vae}$. Therefore, Sobolev's embedding theorem yields  $\m_\vae-(\hat{\m}_{\vae,0}+\hat{\m}_{\vae,-2}\ex^{-i2\psi}+\hat{\m}_{\vae,2}\ex^{i2\psi})=O_1(\sqrt\vae)$. To conclude the proof, we show $\hat{\m}_{\vae,0}=O(\vae)$ and $\hat{\m}_{\vae,\pm 2}=\hat{\m}_{0,\pm 2}+O(\vae)$. Indeed, since (recall \equ{reldh1}) 

$$
R_0^2\m_\vae=R_0^2\cos 2d_\vae(\psi)= R_0^{2}(\cos^2d_\vae-\sin^2d_\vae)= h_\vae^2(\psi+\frac{\pi}{2})-h_\vae^2(\psi)+O_{2}(\vae),
$$
We obtain, taking the average over $[0,2\pi]$, $\hat{\m}_{\vae,0}=O(\vae)$. 

\noi
Finally, for any $k\in \integer$, $|\hat{\m}_{\vae,k}-\hat{\m}_{0,k}|\le \|\m_{\vae}-\m_{0}\|_{C^0}=O(\vae)$. 
\qed

\noi
Denote by $\elips_0$ the ellipse given by the support function
$$
\check {h}_0\coloneqq \su{\sqrt 2}  {R_0} \sqrt{1-(\hat{\m}_{0,-2}\ex^{-i2\psi}+\hat{\m}_{0,2}\ex^{i2\psi})}
\,,$$ 
 by  $\check {d}_0$ its $4$--loop angle function. Let $\check {\m}_0=\cos (2\check {d}_0)$ and $\check {h}_\vae\coloneqq {h}_\vae$. Let $0\le e_0< 1$ be the eccentricity and $\mathfrak{c}_0$ the semi--focal distance of the ellipse $\elips_0$, and by $\check{\g}_0$ its $\psi$--parametrization. 
We are in position to complete the present section by proving that $\O_\vae$ is close  to the ellipse $\elips_0$.

\noi
Observe that (\cf \cite[Eq.~(20)]{bialy2022birkhoff})

\beq{Eq20BilMirhat}
\left\{
\begin{aligned}
&\check  h_0=\check  R_0\sin \check  d_0\,,\\
&\check  h_0'=\check  R_0\cos \check  d_0\, \check  d_0'\,,\\
&\check  h_0''=\check  R_0\cos \check  d_0\, \check  d_0''-\check  R_0\sin \check  d_0 (\check  d_0')^2\,,\\
&\check  d_0(\psi+\frac{\pi}{2})=\frac{\pi}{2}- \check  d_0(\psi)\,.
\end{aligned}
\right.
\eeq
for some $\check {R}_0>0$.

\lem{ClsElpssptF} {\bf (i)} $\check h_\vae=\check  h_0+O_1(\sqrt\vae)$.

\noi
{\bf (ii)} $h_\vae - \check{h}_0 = O_2(\vae^{\frac18})$.
\elem
\proof
{\bf (i)} Indeed, recalling Lemma~\ref{ClsElps}, we have
 \begin{align*}
 h_\vae^2&\eqby{Eq20BilMir} R_0^2\sin^2 d_\vae+O_{2}(\vae)\\
 	&=\frac{R_0^2}{2}(1-\cos 2d_\vae)+O_{2}(\vae)\\
 	&=\frac{R_0^2}{2}(1-\m_\vae)+O_{2}(\vae)\\
 	&=\check  h_0^2+O_1(\sqrt\vae),
 \end{align*}
 hence
\beq{hhathpart}
h_\vae=\check  h_0+O_1(\sqrt\vae).
\eeq 

\noi
{\bf (ii)} We apply Lemma \ref{lem:interpolation} with $f = \check{h}_\vae$, $\delta = \sqrt{\vae}$, $l = m - 1$ to get the claim.
\qed

\lem{gEpschG0}
$\g_\vae\in C^{m - 1}(\torus)$ with $\gamma_\vae = O_{m-1}(1)$ and $\g_\vae-\check{\g}_0 = O_1(\vae^{\frac18})$.
\elem

\proof Since
$$
\g_\vae(\psi)=h_\vae(\psi)\begin{pmatrix}
\cos \psi\\
\sin \psi
\end{pmatrix}+
h_\vae'(\psi)\begin{pmatrix}
-\sin \psi\\
\cos \psi
\end{pmatrix},
\qquad 
$$
For some uniform constant $C$
 \[
		 \|\gamma_\vae\|_{C^{m-1}} \le C \|h_\vae\|_{C^m},
 \]
 and 
\beq{RdsOp}
\|\g_\vae-\check  \g_0\|_{C^1} \le C \|h_\vae - \check{h}_0\|_{C^2} = O(\vae^{\frac18}).
\eeq
\qed

\subsection{Proof of Theorem~\ref{Teo110}  using \cite{kaloshin2018local}}

\noi
For arbitrary eccentricity $0\le e_0<1$ of the ellipse $\elips_0$, we shall apply the following local Birkhoff conjecture result \cite[Main Theorem]{kaloshin2018local}. For, we first introduce the so--called elliptic polar coordinates.

{
		\subsubsection{Elliptic polar coordinates \label{ElipPolCoo}}
}

Assume that the ellipse $\elips_0$ has semi-focal distance $\fc_0$ and eccentricity $e_0$. We consider the elliptic polar coordinates adapted to this family, given by
\[
		\begin{bmatrix}
				x \\ y
		\end{bmatrix}
		= \Phi(\lambda, \theta)
		= \fc_0 
		\begin{bmatrix}
			\cosh \lambda \cos \theta \\ \sinh \lambda \sin \theta
		\end{bmatrix} : \quad
		[0, \infty) \times \T \to \R^2.
\]
  $\Phi$ is a diffeomorphism from $(0, \infty) \times \T$ to its image. In these coordinates, the ellipse $\elips_0$ is given by
\[
		\{(\lambda_0, \theta): \, \theta \in \T\}  
\]
where $e_0 = 1/\cosh(\lambda_0)$. 

Let $\check\gamma_0(\psi)$ be the parametrization for $\elips_0$ in the support function parameter, and as before $\gamma_\vae = \gamma_\vae(\psi)$ is the paremetrization of $\Omega_\vae$, and recall $\gamma_\vae = O_{m-1}(1)$, $\gamma_\vae - \check\gamma = O_1(\vae)$.
\begin{lemma}\label{lem:elliptic-cor}
		There exists $\vae_2 > 0$ depending only on $M$ and $\check\gamma_0$ such that if
		\[
				\|\gamma_\vae - \check\gamma_0\|_{C^1} < \vae_2, 
		\]
		 there exists a $C^m$ function $\lambda_\vae: \T \to \R$ such that $\gamma_\vae$ is represented as the graph of $\lambda_\vae$ in elliptic polar coordinates, i.e.
		\[
				\left\{ (\lambda_\vae(\theta), \theta) : \, \theta \in \T\right\}
				= \{\Phi^{-1} \circ \gamma_\vae(\psi) : \, \psi \in \T\}.
		\]
		Moreover,
		\[
				\lambda_\vae = O_{m-1}(1), \quad
				\lambda_\vae {-\lambda_0} = O_1(\vae^{\frac18}).
		\]
\end{lemma}
\proof
		Assume that $\vae_2$ is chosen such that the curve $\gamma_\vae$ is contained in a compact tubular neighborhood of $\elips_0$. Throughout the proof, $C$ denotes a constant that depends only on $\elips_0$ and $\vae_1$, but may change meaning from line to line.

		Denote 
		\[
				(\bar\lambda_\vae(\psi), \bar\theta_\vae(\psi)) = \Phi^{-1} \circ \gamma_\vae(\psi),   
		\]
		and define $\bar\lambda_0, \bar\theta_0$ similarly for $\check\gamma_0$. Then $\bar\lambda_0(\psi) = \lambda_0$ and $\bar\theta_0(\psi)$ is a diffeomorphism $T \to \T$ with $C^m$ norm depending only on $\check\gamma_0$. Indeed, on $\check\gamma_0$, the change of parameter from $\psi$ to arclength parameter $s$, and the one from $s$ to $\theta$ are both $C^\infty$ smooth. 

		Using smoothness of $\Phi^{-1}$, we have
		\[
				\|\bar\lambda_\vae\|_{C^m}, \|\bar\theta_\vae\|_{C^m} \le C,  \quad
				\|\bar\lambda_\vae - \lambda_0\|_{C^1}, \|\bar\theta_\vae - \bar\theta_0\|_{C^1} \le C \|\gamma_\vae - \check\gamma_0\|_{C^1}. 
		\]
		If $\vae_1$ is small enough, $\bar\theta_\vae(\psi)$ is also a diffeomorphism $\T \to \T$ with $\|\bar\theta_\vae^{-1} - \bar\theta_0^{-1}\|_{C^1} \le C \|\gamma_\vae - \check\gamma_0\|_{C^1} \le C \vae^{\frac18}$, $\|\bar\theta_\vae^{-1} - \bar\theta_0^{-1}\|_{C^m} \le C \|\gamma_\vae - \check\gamma_0\|_{C^m} = O(1)$, where we have applied Lemma \ref{gEpschG0}. Then
		\[
				\lambda_\vae(\theta) = \bar\lambda_\vae \circ \bar\theta_\vae^{-1}(\theta)  
		\]
    verifies the conclusion of the lemma.	 \qed

\subsubsection{Statement of the local Birkhoff conjecture result in \cite{kaloshin2018local} \label{BirkKS}}
\thmtwo{KSLocInt2}{\cite{kaloshin2018local}}
Let $\rwarning{\ell}\ge 39$. For every $M > 0$,
there exists $\d=\d(M)>0$ such that the following holds: if $\O$  is a $3$--rationally \rwarning{weakly} integrable $C^{\rwarning{\ell}}$--smooth billiard table so that $\dpr\O=\elips_0+\l_\O$ with 
$\|\l_\O\|_{C^{\rwarning{\ell}}}\le M$ and $\|\l_\O\|_{C^1}\le \d$ then $\dpr\O$ is an ellipse.
\ethm

\subsubsection{Completion of the proof of Theorem~\ref{Teo110} \label{ProfMainteo}}

\noi
By Lemma~\ref{lem:elliptic-cor}, $\lambda_\vae = O_{\rwarning{\ell}}(1)$ with $\rwarning{\ell}\coloneqq m - 1 \ge 39$, and $\lambda_\vae - \lambda_0 = O_1(\vae^{\frac18})$. Choose $\vae_0$ small enough so that $\vae \in (0, \vae_0)$ implies $\|\lambda_\vae - \lambda_0\|_{C^1} < \delta$, where $\delta$ is from Theorem \ref{KSLocInt2}, we conclude that $\dpr\O_\vae$ is an ellipse. 
\qed

\appendix

\numberwithin{equation}{section}
\numberwithin{theorem}{section}

%

\section{Some technical facts \label{TecfAxt}}
\lem{lem1}
Let $\O$ be a strictly convex planar domain  
and denote by $h$ its support function. Then $\O$ is centrally symmetric with center at the origin \sss $h(\psi+\pi)=h(\psi)$, for all $t\in [0,2\pi]$.

\elem
\proof
Recall \eqref{eqcartsup},  that
$$
\g(\psi)=h(\psi)\begin{pmatrix}
\cos \psi\\
\sin \psi
\end{pmatrix}+
h'(\psi)\begin{pmatrix}
-\sin \psi\\
\cos \psi
\end{pmatrix}.
$$
Then, $\O$ is centrally symmetric \sss $\g(\psi+\pi)+\g(\psi)\equiv0$ which, in turn, is equivalent to $h(\psi+\pi)-h(\psi)\equiv 0$.
\qed

\lem{uNifBoundT0} 
We have $T_0, T_\vae\in C^{\cl-1}$, $T_0, T_\vae = O_{m-1}(1)$, and $\tilde{T}_\vae = T_\vae - T_0 = O_{4}(\vae)$.
\elem
\proof
The formula \eqref{eqcartsup} and $h \in C^{\cl}$ implies $\gamma \in C^{\cl - 1}$. By \cite{douady1982application}, the billiard map is $C^{\cl - 2}$. To upgrade regularity, we note that
 (\textit{cf.} \cite[Proposition~25.5]{gole2001symplectic})
\beq{difBT}
T'(\f,p)=\begin{pmatrix}
-S_{11}\cdot (S_{12})^{-1}  & -(S_{12})^{-1}\\
S_{21}-S_{22}\cdot S_{11}\cdot (S_{12})^{-1} & -S_{22}\cdot (S_{12})^{-1}
\end{pmatrix}\,,
\eeq
we get $T_0', T_\vae' \in C^{\cl - 2}$ and $T_0, T_\vae \in C^{\cl - 1}$. Moreover, $S_\vae - S_0 = O_{5}(\vae)$ and \eqref{difBT} implies $\tilde{T}_\vae = O_{4}(\vae)$.
\qed

\noi
Consider the map $\Psi_\vae: \torus \ni \psi\mapsto \psi' \in \torus $, where $\psi'$ is the $\psi$--parameter corresponding to the next bouncing point of the billiard trajectory in $\O_\vae$ starting at $\psi$ with the angle $d_\vae(\psi)$; we shall denote its lift to $\real$ by $\Psi_\vae$ as well. 
Then,
\lem{lem19}
(i) $\Psi_0(\psi+\pi)=\Psi_0(\psi)+\pi$.

\noi
(ii) 
$\Psi_\vae=\Psi_0+\wt{\Psi}_\vae$ and 
$\Psi_\vae(\psi+\pi)=\Psi_\vae(\psi)+\pi+\wt{\Psi}_\vae(\psi+\pi)-\wt{\Psi}_\vae(\psi)$, with 

\beq{estPsiVaE}
\wt{\Psi}_\vae=O_{3}(\vae).
\eeq
\elem

\proof
(i) This is obvious by the symmetry of $\O_0$ and the uniqueness of the direction which comes back at the starting point after $3$ reflections.

\noi
(ii) 
Observe that
$$
\Psi_\vae= \psi_\vae\circ \pi_1\circ  T_\vae\circ (s_\vae,\cos d_\vae).
$$
Then, Taylor's expansion yields, for some $t_0,t_1,t_2\in (0,1)$,
\[
\wt\Psi_\vae =
\wt\psi_\vae\circ \pi_1 T_\vae(s_\vae,\cos d_\vae) + 
\psi_0'\left (\pi_1 T_0(s_0,\cos d_0) + R_1 \right) 
\cdot \left[\pi_1 T_0'((s_0,\cos d_0) + R_2) \right]
\]
where 
\[
  R_1 = t_2\bigg[\pi_1 T_0'((s_0,\cos d_0)+t_1(\tilde{s}_\vae,-\sin(d_0+t_0\tilde{d}_\vae)\cdot\tilde{d}_\vae)\cdot (\tilde{s}_\vae,-\sin(d_0+t_0\tilde{d}_\vae)\cdot\tilde{d}_\vae)+
	\pi_1\wt T_\vae({s}_\vae,{d}_\vae) \bigg] ,
\]
\[
  R_2 = \big[\pi_1 T_0'((s_0,\cos d_0)+t_1(\tilde{s}_\vae,-\sin(d_0+t_0\tilde{d}_\vae)\cdot\tilde{d}_\vae)\cdot (\tilde{s}_\vae,-\sin(d_0+t_0\tilde{d}_\vae)\cdot\tilde{d}_\vae)
	+\pi_1\wt T_\vae({s}_\vae,\cos {d}_\vae) \big].
\]
Therefore, \equ{estPsiVaE} holds. Moreover, 
\begin{align*}
\Psi_\vae(\psi+\pi)&=\Psi_0(\psi+\pi)+\wt \Psi_\vae(\psi+\pi)\\
	&\overset{(i)}=\Psi_0(\psi)+\pi+\wt \Psi_\vae(\psi+\pi)\\
	&=\Psi_\vae(\psi)+\pi+\wt{\Psi}_\vae(\psi+\pi)-\wt{\Psi}_\vae(\psi).\qedeq
\end{align*}

{ 
		\section{Interpolation--type result \label{interprEz}}
}

\begin{lemma}\label{lem:interpolation}
		Suppose $f \in C^l(\T)$ where $l \ge 6$, $\|f\|_{C^l} = C > 0$ and $\|f\|_{C^1} = \delta$. Then
		\[
				f = O_2(\delta^{\frac14}).  
		\]
\end{lemma}
\proof
Choose $\sigma = \frac72$, we estimate the Sobolev norm $\|f\|_{H^\sigma}$ which bounds $\|f\|_{C^2}$. Let $\hat{f}_k$ be the Fourier series of $f$, then by the Hausdorf-Young inequality
\[
		|k|^l|\hat{f}_k| \le \|(f^{(l)})^{\hat{}}\|_{L^\infty}
		\le \|f^{(l)}\|_{L^1} \le \|f\|_{C^l}, 
\]
we obtain (write $\langle f\rangle = \max\{1, |k|\}$)
\[
		|\hat{f}_k| \le C\langle k\rangle^{-l}
\]
and similarly
\[
		|\hat{f}_k| \le \delta \langle k\rangle^{-1}.
\]

For $a > 0$ to be determined,
\[
		\sum_{\langle k\rangle \ge \delta^{-a}} \langle k\rangle^{2\sigma} |\hat{f}_k|^2
		\le C^2 \sum_{\langle k\rangle \ge \delta^{-a}} \langle k\rangle^{2(\sigma - l)}
		\le C^2  \delta^{2a} \sum_{\langle k\rangle \ge \delta^{-a}} \langle k\rangle^{2(\sigma - l + 1)} \le C_1 \delta^{2a}
\]
for some $C_1 > 0$, since $\sigma - l + 1 = \frac{9}{2} - l \le - \frac{3}{2}$. On the other hand, 
\[
		\sum_{\langle k\rangle < \delta^{-a}} \langle k\rangle^{2\sigma} |\hat{f}_k|^2
		\le \delta^2 \sum_{\langle k\rangle < \delta^{-a}} \langle k\rangle^{2(\sigma - 1)}
		\le 2 \delta^2  \delta^{-a} \cdot \delta^{-2a(\sigma - 1)} 
		= 2 \delta^{2 + a - 2a\sigma}. 
\]
Set $2a = 2 + a - 2a\sigma$, we get $a = \frac{2}{1 + 2\sigma} = \frac{1}{4}$. It follows that
\[
		\|f\|_{C^2} \le	\|f\|_{H^\sigma} = \left( \sum_k \langle k\rangle^{2\sigma} |\hat{f}_k|^2\right)^{\frac12} = O(\delta^a) = O(\delta^{\frac14}).
\]
\qed

%
{
		\section{Uniqueness of the $(p, q)$-loop orbits} \label{sec:d0}
}

We will begin with the general setting of twist maps and specialize to billiards later in this section.

Consider a $C^2$ generating function $H: \real \times \real \to \real$ which satisfies the following conditions:
\begin{enumerate}[(1)]
  \item $H(x + 1, x' + 1) = H(x, x')$.
  \item There exists $M > 0$ such that $\sup_{x, x'} |\partial_{ij} H(x, x')| < M$, $i,j \in {1, 2}$. 
  \item There exists $\rho > 0$ such that $\partial_{12} H(x, x') < - \rho$ for all $x, x'$.
\end{enumerate}
We remark that condition (3) implies the traditional twist ($x' \mapsto \partial_2 H(x, x')$ is monotone) and superlinearity ($\lim_{|x' - x| \to \infty} H(x, x')/|x' - x| = \infty$) conditions. In fact $H$ satisfies the conditions defined by Mather (see \cite{mather1994action}), which implies all the classical results of Aubry-Mather theory.
Then $H$ defines an exact area preserving twist map $F = F(x, r)$ on $\real \times \real$ via
\begin{equation}  \label{eq:twist-map-gen}
  F(x, r) = (x', r') \quad \iff \quad
  r = - \partial_1 H(x, x'), \ \ \ r' = \partial_2 H(x, x'),
\end{equation}
which also projects to a map of $\torus \times \real$. Note that
\[
		\partial_{12} H(x, x') = - \frac{\partial r}{\partial x'} < - \rho < 0,  
\]
which implies the mapping $r \mapsto x' = \pi_1 F(x, r)$ is an increasing global diffeomorphism. Moreover, $\|DF\|$, $\|DF^{-1}\|$ are bounded by a constant depending only on $\|D^2H\|_{C^0}$ and $\rho$. 

An orbit $(x_k, r_k)$ of $F$ on $\real \times \real$ is uniquely determined by the sequence $(x_k)$, and  any orbit $(x_k, r_k) \in \torus \times \real$ admits a unique lift $(\tilde{x}_k)$ once a lift of $x_0$ is chosen. 

We say $(x, r)$ is $(p, q)$-periodic if for the lifted map $F$, 
\[
  F^q(x, r) = (x + p, r).  
\]
We say the orbit of $(x, r)$ is a \emph{$(p, q)$-loop} orbit if for the lifted map $F$,
\[
  \pi_1 F^q(x, r) = x + p.  
\]
$(p, q)$-loop orbits aren't necessarily periodic on $\T \times \R$.

An invariant curve $\gamma$ of $F$ is called \emph{essential} if it is not homotopic to a point. $\gamma$ is said to be $(p, q)$-integrable if every orbit on $\gamma$ is $(p, q)$-periodic. We have the following classical theorem by Birkhoff (\cite{Bir22}), and the proof of the Lipschitz constant can be found in \cite{herman1983courbes}.

\begin{theorem}[\cite{Bir22, herman1983courbes}]
		Any essential invariant curve of a twist map is a Lipschitz graph over $\T$. Moreover, the Lipschitz constant depends only on $\|D^2 H\|_{C^0}$ and $\rho$.	
\end{theorem}

\begin{proposition}\label{prop:pq-loop-map}
		Suppose the twist map $F$ admits a $(p, q)$-rationally-integrable curve $\gamma$. Then there is $\vae_3 > 0$, $c > 0$, depending only on $q$, $\|D^2 H\|_{C^0}$, $\|D^3H\|_{C^0}$, and $\rho$ such that the following hold.

		If $F_1$ is a twist map whose generating function $H_1$ satisfies $\|H_1 - H\|_{C^2} < \vae_3$, then (after lifting to $\R \times \R$) the equation
		\[
		  F_1^q(x, r) = x + p
		\]
		has a unique solution $\bar{r}(x)$ for every $x$. Moreover,
\begin{equation}  \label{eq:Fq-twist}
		|\partial_r F_1^q(x, \bar{r}(x))| > c. 
\end{equation}
\end{proposition}

\lem{twist}
Let $\gamma$ be a continuous invariant graph of the twist map $F: \torus \times \real \to \torus \times \real$, and suppose $\torus \times \real \setminus \gamma = U^+ \cup U^-$, where $U^\pm$ are the components above and below $\gamma$. Then both $U^\pm$ are invariant. 

There exists $c > 0$ depending only on $\|D^2 H\|_{C^0}$, such that if $(x_k, r_k)_{k \in \mathbb{Z}}$ is an orbit in either $U^+$ or $U^-$, for all $j < k$, we have
\begin{equation}  \label{eq:quant-twist}
\begin{aligned}
\pi_1 F^{k - j}(V(x_j) \cap \gamma) - x_k
> c^{k-j} \rho \vdist((x_j, r_j), \gamma),
 & \quad \text{ if } \{(x_j, r_j)\}  \subset U^-, \\
\pi_1 F^{k - j}(V(x_j) \cap \gamma) - x_k
< - c^{k-j} \rho \vdist((x_j, r_j), \gamma),
& \quad \text{ if } \{(x_j, r_j)\}  \subset U^+, \\
\end{aligned}
\end{equation}
where
\[
  \vdist((x,r), \gamma) = \dist((x,r), V(x) \cap \gamma).  
\]

the comparison is done after lifting $F$ to $\R \times \R$.
\elem
\proof
  The twist property implies that for all $x \in \real$, the curve $F(V(x))$ is a monotonically increasing curve. This implies for any $(x, r)$ above $\gamma$, the image $F(x, r)$ will still be above $\gamma$. The invariance of $U^\pm$ follows.

  We only prove the lemma in the case that the orbit is contained in $U^-$, as the other case is similar. Let $(x_k, \delta_k) = V(x_k) \cap \gamma$, then by assumption, $\delta_k > r_k$. By the twist propery, 
\[
		\pi_1 F(x_k, \delta_k) - x_{k+1} \ge \int_{r_k}^{\delta_k} \partial_r \pi_1 F(x, r)dr > \rho (\delta_k - r_k) > 0, \quad \forall k \in \mathbb{Z}.
\]
In particular, $\pi_1 F(x_j, \delta_j) > x_{j+1}$.

Since $\gamma$ is an invariant graph, the mapping
\[
  g_\gamma: x \mapsto \pi_1 F(V(x) \cap \gamma)  
\]
is an orientation preserving diffeomorphism on $\T$. Moreover, since $\gamma$ is Lipschitz with constant depending only on $\|D^2 H\|_{C^0}$ and $\rho$, the Lipschitz constant of $g_\gamma$ and $g_\gamma^{-1}$ depends only on the same constants. 

For the rest of the proof, $F$ stands for the lifted map to $\R \times \R$.
Let $c = 1/\|D g_\gamma^{-1}\|_{C^0} \in (0, 1]$, then for all $y > x \in \R$ 
\begin{equation}  \label{eq:inv-curve-exp}
  \pi_1 F(V(y) \cap \gamma) - \pi_1 F(V(x) \cap \gamma) > c (y - x) > 0.
\end{equation}
We first prove by induction that
\begin{equation}  \label{eq:monotone}
		\pi_1 F^{k - j}(x_j, \delta_j)	- x_k > 0.
\end{equation}
Indeed, assume by induction that $\pi_1 F^{k - j - 1} (x_j, \delta_j)	- x_{k-1} > 0$, we have
\[
		\pi_1 F^{k - j}(x_j, \delta_j) 
		= \pi_1 F(V(\pi_1 F^{k - j -1}(x_j, \delta_j)) \cap \gamma)
		> \pi_1 F(V(x_{k-1}) \cap \gamma)
		> x_k.
\]

We now have
\[
  \begin{aligned}
			\pi_1 F^{k - j}(x_j, \delta_j)
			& > \pi_1 F^{k - j - 1}(V(x_{j + 1} + \rho(\delta_j - r_j))) \\
			& > c^{k - j - 1} \rho(\delta_j - r_j) + \pi_1 F^{k - j - 1}(x_{j+1}, \delta_{j+1})
			> c^{k - j} \rho(\delta_j - r_j) + x_k,
  \end{aligned}  
\]
where the second inequality is by applying \eqref{eq:inv-curve-exp} $k - j- 1$ times, and the third is due to \eqref{eq:monotone}.
\qed

\proof
(Proof of Proposition \ref{prop:pq-loop-map})

		Lemma \ref{twist} implies that if $(x, r) \in U^-$ (resp $U^+$), then for the lifted map $F$, $\pi_1 F^q(x, r) < x + p$ (resp. $\pi_1 F^q(x, r) > x + p$). In other words, for any given $x$, the unique solution to 
		\[
		   \pi_1 F^q(x, r) = x + p
		\]
		is given by the condition $(x, r) \in \gamma$.

		For the ``moreover'' part, we note that $\pi_1 F^q(x, r) - (x + p)$ for $(x, r) \in \gamma$, and \eqref{eq:quant-twist} implies
\begin{equation}  \label{eq:Fq-twist-proof}
\partial_r \pi_1 F^q(x, r) > c^q \rho.	
\end{equation}
There exists $r_0 > 0$ depending only on $q$ and $\|D^2 F\|_{C^0}$ (hence $\|D^3 H\|_{C^0}$), such that
\begin{equation}  \label{eq:Fq-derivative}
		\partial_r \pi_1 F^q(x, r) > c^q \rho/2, \text{ if } \vdist((x, r), \gamma) < r_0.
\end{equation}
Moreover, from \eqref{eq:quant-twist},
\begin{equation}  \label{eq:Fq-boundary}
\begin{aligned}
	& F^q(V(x) \cap \gamma + (0, r))  - (x + p) > c^q\rho r_0,  & r \ge r_0,  \\
	&  F^q(V(x) \cap \gamma - (0, r))  - (x + p) <  - c^q\rho r_0, & r \ge r_0.
\end{aligned}
\end{equation}

For any $F_1$ that is $\vae_0$-$C^1$-close to $F$ with $\vae_0$ depending only on $r_0$ and $c^q \rho$, \eqref{eq:Fq-derivative} and \eqref{eq:Fq-boundary} holds for $F_1$, as long as we replace $c^q \rho$ with $c^q \rho/2$. These equations imply uniqueness of the solution 
\[
  \pi_1 F^q(x, r) - (x + p)   = 0.
\]
Finally, \eqref{eq:Fq-twist} follows from \eqref{eq:Fq-derivative}.
\qed

We now specialize to the billiard case. Given a billiard boundary $\Omega$, define the generating function
\[
  H_\Omega(s_1, s_2) = - l(s_1, s_2)  
\]
for all $s_1, s_2 \in \R$ such that $0 \le s_2 - s_1 \le 1$. Then the billiard map can be written as
\[
  F(s_1, r_1) = (s_2, r_2) \iff 
	r_1 = - \partial_1 H_\Omega(s_1, s_2),   
	\quad
	r_2 = \partial_2 H_\Omega(s_1, s_2), 
\]
where the relation between the $r$ variable and the reflection angle $\theta$ is $r = - \cos \theta$ (see Figure \ref{figgBM}). One checks that the billiard map is defined on $\T \times (-1, 1)$ extensible continuously to $[-1, 1]$.

While the billiard map is a twist map in this sense, it has vanishing twist at the boundary. Indeed, if 
\[
   F(s_1, r_1) = (s_2, r_2), \quad
	 r_1 = - \cos \theta_1, \quad
	 r_2 = - \cos \theta_2,
\]
where $\theta_1(s_1, s_2) = \angle (\dot\gamma(s_1), \gamma(s_2) - \gamma(s_1)), \theta_2(s_1, s_2) = \angle (\gamma(s_2) - \gamma(s_1), \dot\gamma(s_2)) \in [0, \pi]$, 
then
\begin{equation}  \label{eq:H12-billiard}
		\partial_{12} H(s_1, s_2) = \frac{\partial r_2}{\partial s_1} 
		= - \frac{\partial r_1}{\partial s_2}
		=  \sin \theta_2 \frac{\partial \theta_2}{\partial s_1}
		= - \frac{\sin \theta_2}{\sin \theta_1} l(s_1, s_2),
\end{equation}
where the last identity $\frac{\partial \theta_2}{\partial s_1} = - \frac{l}{\sin \theta_1}$ can be found in \cite{katok2006invariant}, (V.4.9). As $s_2 - s_1 \to 0$, \eqref{eq:H12-billiard} converges to $0$, as seen from the next lemma.

\begin{lemma}\label{lem:uniform-est-H12}
Suppose the boundary $\gamma$ is normalized to arclength $1$ with arclength parametrization, and assume the curvature bound
\[
  0 < \kappa_0 \le \kappa(s) \le \kappa_1.  
\]
Then there exists $C > 1$ depending only on $\kappa_0, \kappa_1$ such that
\[
\begin{aligned}
		& C^{-1}  \dist_{\T}(s_1, s_2) \le l(s_1, s_2) \le C \dist_\T(s_1, s_2),  \\
		& C^{-1}   \dist_{\T}(s_1, s_2) \le \theta_1, \theta_2 \le C \dist_\T(s_1, s_2),  \\
		&  C^{-1} \dist_\T(s_1, s_2) \le |\partial_{12} H(s_1, s_2)| \le C \dist_\T(s_1, s_2).
\end{aligned}
\]
for all $s \in \T$. 
\end{lemma}
\begin{proof}
		Thoughout the proof, we write $f \lesssim g$ if $f \le C g $ for a constant depending only on $\kappa_0$, $\kappa_1$, and $f \approx g$ if $f \lesssim g$ and $g \lesssim f$ .

		For a fixed $s_1 \in \R$, consider the function
		\[
		  \alpha(s) = \angle (\dot\gamma(s), \dot\gamma(s_1) )
		\]
		defined which maps $[s_1, s_1 + 1]$ diffeomorphically to $[0, 2\pi]$.

		Perform a rotation so that $\dot\gamma(s_1) = (1, 0)$, then $\gamma(s_2) - \gamma(s_1) = \int_{s_1}^{s_2}  \bmat{\cos \alpha(s) \\ \sin \alpha(s)} ds $. This mean in general 
		\[
				l(s_1, s_2) = |\gamma(s_2) - \gamma(s_1)| = \left| 
				\int_{s_1}^{s_2} \bmat{\cos \alpha(s) \\ \sin \alpha(s)} ds \right|.   
		\]
		Let
		\[
				S_1 = \{s \in [s_1, s_1 + 1] :\,  \alpha(s) \in [0, \frac{\pi}{4}] \cup [2\pi - \frac{\pi}{4}, 2\pi]\}, \quad
				S_2 = \overline{[s_1, s_1 + 1] \setminus S_1}.
		\]
		If $s \in S_1$, notting $\cos \alpha(s) \ge 0$, we have
		\[
				l(s_1, s_2) \ge \int_{s_1}^{s_2} \cos \alpha(s) ds \ge  \min_{s \in [s_1, s_2]} \cos \alpha(s) \dist_{\T}(s_1, s_2) \gtrsim (s_2 - s_1). 
		\]

	For any $t_1 < t_2$
		\[
				\alpha(t_2) - \alpha(t_1)  = \int_{t_1}^{t_2} \kappa(s) ds \approx t_2 -  t_1,   
		\]
		we have
		\[
		   \frac{\pi}{4} \le \min\{\alpha(s_*), 2\pi - \alpha(s_*)\}  
			 \lesssim \min\{s_* - s_1, s_1 + 1 - s_*\} = \dist_\T(s_1, s_*).
		\]
			Let $s_*$ denote the $s \in [s_1, s_1 + 1]$ which minimimizes $l(s_1, s)$ on the set over all $s \in S_2$. Then either $\alpha(s_*) \in [\frac{\pi}{4}, \pi]$ or $\alpha(s_*) \in [\pi, 2\pi - \frac{\pi}{4}]$. Assume the former as the other case is similar. For  $s_2 \in S$,
		\[
\begin{aligned}
		l(s_1, s_2)  & \ge l(s_1, s_*) \ge \int_{s_1}^{s_*} \sin \alpha(s) ds
				\ge \int_{s_1 + \pi/8}^{s_1 + \pi/4} \sin \alpha(s) ds
				\gtrsim \dist_\T(s_1, s_*) \\
								 & \gtrsim \frac{\pi}{4} \gtrsim \dist_\T(s_1, s_2)
\end{aligned}
		\]
		since $\dist_\T(s_2 - s_1) \le \frac12$. We have proven the bound $l(s_1, s_2) \gtrsim \dist_\T(s_1, s_2)$. The other bound is trivial since chord length is always smaller than arclength: $l(s_1, s_2) \le \dist_\T(s_1, s_2)$. The first inequality follows.

		For the second inequality, let $\theta_1, \theta_2$ be as in \eqref{eq:H12-billiard}. Rotate the axis so that $\gamma'(s_1) = (1, 0)$, we have
		\[
				\sin \theta_1 = \frac{\int_{s_1}^{s_2} \sin \alpha(s) ds}{l(s_1, s_2)}
				= - \frac{\int_{s_2}^{s_1 + 1} \sin \alpha(s) ds}{l(s_1, s_2)}.
		\]
		If $s_2 \in S_1$, then
		\[
		  |\sin \theta_1|  \approx \frac{\dist_\T(s_1, s_2)^2}{l(s_1, s_2)} \approx \dist_\T(s_1, s_2).
		\]
		If $s_2 \in S_2$, we let $s_*$ minimize $\sin \theta_1$ over $S_2$, argue similarly as before to get
		\[
		  |\sin \theta_1| \gtrsim \frac{\dist_\T(s_1, s_*)^2}{l(s_1, s_2)}  
			\gtrsim \dist_\T(s_1, s_2).
		\]
		The upper bound $|\sin \theta_1| \lesssim \dist_\T(s_1, s_2)$ is trivial since $\dist_\T(s_1, s_2) \gtrsim 1$ when $s_2 \in S_2$. Moreover, the same estimate works for $\theta_2$ by symmetry. By \eqref{eq:H12-billiard}, we get
		\[
				|\partial_{12} H(s_1, s_2)| \approx \dist_\T(s_1, s_2)  
		\]
		as required.
\end{proof}

\begin{proposition}\label{prop:q-loop-billiard}
Let $\Omega$ be a strictly convex billiard table and let $\gamma(s)$ be the arclength parametrized boundary, and $\kappa(s)$ the curvature function, we assume that there is $\kappa_0 > 0$ and $C > 0$ such that
\[
		\kappa(s) > \kappa_0, \quad |\partial \Omega| < C, \quad \|\gamma\|_{C^3} < C.  
\]
Suppose the billiard map $T$ admits a $(p, q)$-rationally integrable caustic. 

Then there is $\vae_4 > 0$, $c > 0$, depending only on $q$, $C$, and $\kappa_0$, such that if billiard table $\Omega_1$ whose boundary $\gamma_1$ satisfies $\|\gamma_1 - \gamma\|_{C^2} < \vae_4$, for the lifted billiard map $T_1$, the equation
\[
  \pi_1 F_1^q(s, r) = s + p
\]
has a unique solution $\bar{r}(s)$, with
\[
		|\partial_r \pi_1 F_1^q(s, \bar{r}(s))| > c, \quad |\bar{r}(s) \pm 1| > c. 
\]
\end{proposition}

\proof
Without loss of generality, assume $|\partial \Omega|$ is normalized to $1$. We continue the use of notations $\lesssim$ and $\approx$, where the constant factor depends on $q$, $C$, and $\kappa_0$.

Suppose $F(s_1, r_1) = (s_2, r_2)$, then Lemma \ref{lem:uniform-est-H12} implies
\[
		\left| \partial_r \pi_1 F(s_1, r_1) \right| \approx \frac{1}{|\partial_{12} H(s_1, s_2)} \approx \dist_\T(s_1, s_2)^{-1}.
\]
Assume that $s_2 - s_1 \in (0, \frac12)$ so that $\dist_\T(s_1, s_2) = s_2 - s_1$, then 
\[
\begin{aligned}
  \pi_1 F(s_1, r_1) - s_1
	& = \pi_1 F(s_1, r_1) - \pi_1 F(s_1, -1)
	= \int_{-1}^{r} \partial_r \pi_1 F(s_1, \sigma) d\sigma \\
	& \lesssim |r + 1| (s_2 - s_1)^{-1}
	= 2\sin^2 \frac{\theta_1}{2} (s_2 - s_1)^{-1}
	\approx s_2 - s_1,
\end{aligned}
\]
noting that $\theta_1 \approx s_2 - s_1$. If $(s_k, r_k)$, $0 \le k \le q$ is a $(p, q)-$loop orbit, then
\[
		p = s_q - s_0 \approx s_k - s_{k-1}  
\]
for any $1 \le k \le q$, where the constant factor may depend on $q$. Therefore, there exists a constant $\delta_0 > 0$ such that
\[
  s_k - s_{k-1} > \delta_0, \quad 1 \le k \le q  
\]
for any $q$-loop orbit. Since $|\bar{r}(s) + 1| = |r_0 + 1| \approx (s_1 - s_0)^2$, there exists $c > 0$ such that $|\bar{r}(s) + 1| > c$. When $s_2 - s_1$ is close to $1$, a symmetric argument yields $|\bar{r}(s) - 1 | > c$. 

The billiard map $F$ is uniformly twist on the part of the phase space where $1 - \delta_0 > s_2 - s_1 > \delta_0$. Modifying the generating function $H_\Omega(s_1, s_2)$ on the set $s_2 - s_1 \le \delta_0$ and $s_2 - s_1 \ge p - \delta_0$ so that the modified generating function to $H$ is uniformly twist and defined on $\R \times \R$, with the map  denoted $\tilde{F}$. We now apply Proposition \ref{prop:pq-loop-map} to get uniqueness of $(p, q)$-loop orbit for the map $\tilde{F}$ and any small $C^3$ perturbation $H_1$ to the generating function $H$. Our proposition now follows since any $q$-loop orbit of $F$ must be a $(1, q)$-loop orbit of $\tilde{F}$, and a $C^3$-small perturbation to $\gamma$ results in a $C^3$-small perturbation to the generating function $H$.
\qed

Finally, we are ready to prove Propisition \ref{prop:4-loop}.

\proof (Proof of Proposition \ref{prop:4-loop})
Choose  $\vae_5 > 0$ such that all the curves in
\[
		\cV_1 =  \{\gamma:  \|\gamma - \gamma_0\|_{C^2} < \vae_1, \quad \|\gamma\|_{C^4} < C\},  
\]
all admits the bounds
\[
		\kappa(s) > \underline{\kappa}/2.
\]

For a curve $\gamma = \gamma(\psi)$ parametrized using the normal angle $\psi$, let $\psi = g_\gamma(s)$ be the coordinate change to the arclength coordinates. The coordinate changes $g_\gamma$ and $g_\gamma^{-1}$ admits uniform $C^3$ bounds depending only on $\kappa_0$ and $C$. Therefore there exists $C_1 > 1$ such that
\[
		C_1^{-1} \|\gamma - \gamma_0\|_{C^2} \le \|\gamma \circ g_\gamma - \gamma_0\circ g_{\gamma_0} \|_{C^2} < C_1 \|\gamma - \gamma_0\|_{C^2}
\]
and
\[
		\|\gamma\circ g_\gamma\|_{C^3} < C_1  
\]
for all $\gamma \in \cV_1$.

Let $\vae_4$ be the small parameter given in Proposition \ref{prop:q-loop-billiard}, which depends only on $\underline{\kappa}$, $q$, $C$ and $C_1$. In particular, if $\gamma_\vae \in \cV_1$ admits $4$-rationally-integrable caustic, then any arclength parametrized curve $\eta$ satisfying
\[
		\|\eta - \gamma_\vae \circ g_{\gamma_\vae}\|_{C^2} < \vae_4
\]
admits unique $4$-loop orbits. Define $\vae_1 = \min\{\vae_4, \vae_5\}$, and
\[
		\cV = \{\gamma: \|\gamma - \gamma_0\|< C_1^{-1} \vae_0, \, \|\gamma\|_{C^4} < C \},
\]
then if $\gamma_\vae \in \cV$ admits a $4$-caustic, for any other $\gamma \in \cV$, we have
\[
		\|\gamma \circ g_{\gamma} - \gamma_\vae \circ g_{\gamma_\vae}\|_{C^2}
		\le C_1 \|\gamma - \gamma_\vae\|_{C^2} < \vae_4,  
\]
hence $\gamma$ admits unique $4$-loop orbits.
\qed

\proof (Proof of Lemma \ref{lem:d-estimates})
Let $F_0(s, r)$, $F_\vae(s, r)$ denote the billiard map for $\Omega_0$, $\Omega_\vae$ in the traditional coordinates, and let $\psi = g_0(s)$, $\psi = g_\vae(s)$ the coordinate change between $\psi$ and $s$. In the notation of Proposition \ref{prop:q-loop-billiard}, 
\[
	- \cos	d_0(\psi) = \bar{r}_0 \circ g_0^{-1}(\psi),   \quad
   - \cos		d_\vae(\psi) = \bar{r}_\vae \circ g_\vae^{-1}(\psi).
\]
Then there exists $c > 0$ such that $|\bar{r}_0 \pm 1| > c$, $|\bar{r}_\vae \pm 1| > c$, or equivalently, $0 < \underline{d} < d_0, d_\vae < \overline{d} < \pi$ for some $\underline{d}, \overline{d}$.

By Proposition \ref{prop:q-loop-billiard}, $\bar{r}$ is the unique solution to 
\[
  F^4(s, r) = s + 1,
\]
and since $|\partial_r \pi_1 F^4(s, r)|> c > 0$, the implicit function applies. When $F$ is $C^{m-1}$, so is $\bar{r}$ with $\bar{r} = O_m(1)$. The implicit function theorem also implies
\[
		\|\bar{r}_\vae - \bar{r}_0\|_{C^3} \lesssim \|F_\vae - F_0\|_{C^3} \lesssim \|\gamma_0 - \gamma_\vae\|_{C^4} \lesssim \|h_0 - h_\vae\|_{C^5}. 
\]
\qed

\section{Relation between rational and $C^0$ integrability \label{EquivInteg}}
%
%

Let $p_1/q_1 < p_2/q_2 \in \rational \cap (0, \infty)$, we say the billiard map in the domain $\Omega$ is rationally integrable in the interval $[p_1/q_1, p_2/q_2]$ if for every rational number $\rho \in [p_1/q_1, p_2/q_2] \cap \rational$, there exists a smooth, strictly convex caustic on which the dynamics is conjugate to a rigid rotation with rotation number $\rho$. Same as in the previous section, we use the generating function $H_\Omega(s_1, s_2) = - l(s_1, s_2)$, and the billiard map is given by 
\[
  F(s, r) = (s_1, r_1), 
  \quad \iff \quad
  r =  - \partial_1 H_\Omega(s, s_1), \quad r_1 =   \partial_2 H_\Omega(s, s_1),
\]
with $r = - \cos \theta$. The map is defined on $\torus \times (-1, 1)$ extending continously to $\torus \times \{-1, 1\}$. In this point of view, rotation numbers of invariant curve can range from $0$ to $1$ (with the counter clockwise rotation counts as rotation number $[1/2, 1]$).

A homotopically non-trivial invariant curve shall be called an essential invariant curve.

We will prove the following statement.

\pro{prop:rat-int}
Suppose the billiard admit two essential invariant curves $\gamma_1, \gamma_2$ on which the dynamics are conjugate to rigid rotations of rotation numbers $\rho_1 < \rho_2 \in \mathbb{Q}$. Then the billiard is rationally integrable on $[\rho_1, \rho_2]$ if and only if it is $C^0$-integrable on the phase space between $\gamma_1$ and $\gamma_2$.
\epro

\noi
Proposition~\ref{prop:rat-int} follows from a more general result about twist maps. The assumption for the generating function is the same as in the previous section.

\begin{definition} 
		We say an orbit $\{(x_k, r_k)\}_{k \in \mathbb{Z}}$ of $F$ has no conjugate points if for all $j < k$, we have
\[
  \frac{\partial (\pi_1 F^{k - j}(x, r))}{\partial r} \Bigr|_{(x_j, r_j)} \ne 0,
\]
where $\pi_1(x, r) = x$. 
\end{definition}

{
\pro{prop:rat-int-twist}
Suppose $H$ satisfies the conditions (1) - (3) stated above. Let $\gamma_1$ and $\gamma_2$ be two essential invariant curves of $F$, on which the dynamics are conjugate to rigid rotations of rotation numbers $\rho_1 < \rho_2 \in \mathbb{Q}$. Then the following are equivalent.

  \begin{enumerate}[(a)]
    \item $F$ is rationally integrable on $[\rho_1, \rho_2]$.
    \item $F$ has no conjugate points for all the orbits in the phase space between $\gamma_1$ and $\gamma_2$.
    \item $F$ is $C^0$-integrable on the phase space between $\gamma_1$ and $\gamma_2$.
\end{enumerate}
\epro
}

\proofff[Proof of Proposition~\ref{prop:rat-int} using Proposition~\ref{prop:rat-int-twist}]
  Suppose the billiard table is normalized to perimeter $1$. The generating function $H_\Omega(s_1, s_2)$	can be extended to the set $s_2' - s_1' \in [0, 1]$ in $\real \times \real$, and for any $\epsilon > 0$,  $\partial_{12}H_\omega < -\rho(\epsilon) < 0$ over $s_2' - s_1' \in [\epsilon, 1 - \epsilon]$. The generating function admits a smooth, periodic extension to $\real \times \real$ keeping the estimate $\partial_{12} H < -\rho(\epsilon)$  and uniform bound on the second derivatives (see \cite{mather1994action}, section 8). There exists $\delta(\epsilon) > 0$ satisfying  $\lim_{\epsilon \to 0} \delta(\epsilon) = 0$, such that extended map coincide with the billiard map over the set $\torus \times [-1 + \delta(\epsilon), 1 - \delta(\epsilon)]$. Finally, our proposition follows by choosing $\epsilon$ small enough such that the phase space between the two rational invariant curves are contained in the set $\torus \times [-1 + \delta(\epsilon), 1 - \delta(\epsilon)]$, and applying Proposition \ref{prop:rat-int-twist} to the extension system.
\qed

\noi
For the case of Tonelli Hamiltonians on $\torus^n \times \real^n$ and under the assumption that there is a totally periodic Lagrangian invariant graph for every rational vector, Proposition \ref{prop:rat-int-twist} is essentially proven in section 2.2 and 2.3 of \cite{AABZ15}. We provide a proof here because the twist map case cannot be directly reduced to the Tonelli case (although the proofs are very similar), and also the notion of rationally integrable on an interval is purely one-dimensional. We mostly follow the arguments of \cite{AABZ15}, and also mention \cite{AMS22} where some analogous statements in higher dimensional twist maps are proven.

We will start with the implication (a) $\Rightarrow$ (c).

For $c \in \real$, define $A_c: \torus \times \torus \to \real$ by
\[
  A_c(x, y) = \min \left\{ H(x', y') - c (y' - x'): \quad x' = x,\, y' = y \mod 1\right\}. 
\]
Given $x_0', \cdots, x_n' \in \real$, let us denote
\[
  H\left( (x_k')_{k = 0}^n\right) = \sum_{k = 0}^{n-1} H(x_k', x_{k+1}'), 
\]
and define $A_c^n: \torus \times \torus \to R$ by
\[
  \begin{aligned}
    A_c^n(x, y) & = \min \left\{ \sum_{k = 0}^{n-1} A_c(x_k, x_{k+1}): 
  \quad x_0 = x, \, {x_n} = y\right\}  \\
		& = \min \left\{ H((x_k')_{k = 0}^n):
		\quad x_0' = x, \, x_n' = y \mod 1\right\}.
  \end{aligned}
\]
The \emph{Lax-Oleinik} operator $T: C(\torus) \to C(\torus)$ is defined as
\[
  T_c u(x) = \min_{y \in \torus} \left\{ u(y) + A_c(y, x)\right\} .
\]

\noi
We have the following standard results from weak KAM theory.
\protwo{prop:weak-KAM}{see \cite{Fat05, Zav10}}
  \begin{enumerate}[(1)]
 \item All $A_c^n(x, y)$ are equi-Lipschitz in both variables. 
 \item There exists unique $\alpha(c) \in \real$ called Mather's alpha function such that all $A_c^n(x, y) + n\alpha(c)$ are equibounded. In particular, $-\alpha(c) = \lim_{n \to \infty} \frac{1}{n} A_c^n(x, y)$.
 \item $\alpha(c)$ is a convex function in $c \in \real$.
 \item (Weak KAM Theorem) There exists $u \in C(\torus)$ called a weak KAM solution such that $T_c u + \alpha(c) = u$. 
 \item If $u$ is a weak KAM solution, then for every $x \in \torus$, there exists an $F$-orbit $(x_k, r_k)_{k = -\infty}^0$ with $x_0 = x$, such that
   \[
     A_c^n(x_{-n}, x_0) + u(x_{-n}) = \sum_{k = -n}^0 A_c(x_k, x_k+1) + u(x_{-n}) = u(x_0). 
   \]
   This orbit is called a calibrated orbit for $u$.
\end{enumerate}	
\epro

\noi
Birkhoff's Theorem says any continuous invariant graph $\gamma$ of $F$ must be Lipschitz. Then there exists a $C^{1,1}$ function $u$ and $c \in \real$ such that $\gamma = \{(x, c + du(x)) : \, x \in \torus\}$.

\lem{lem:periodic-minimum}
  Let $u \in C^1(\torus)$ and $\gamma = \{(x, du(x)): \, x \in \torus\}$ is invariant under $T$. Then $u$ is a weak KAM solution:
  \[
    T_c u + \alpha(c) = u. 
  \]

  Suppose in addition that $\gamma$ is a $p/q$ periodic invariant curve. Let $(x_k', r_k')_{k = 0}^q$ the lift of any orbit on $\gamma$ to $\real \times \real$. Then 
  \begin{equation}  \label{eq:pqmin}
    H\left( (x_k')_{k = 0}^q\right)
    = \min \left\{ H\left( (y_k')_{k = 0}^q\right)
    : \, y_0' = x_0, \, y_q' = x_q' = x_0' + p\right\}.
  \end{equation}
  Conversely, any sequence satisfying \eqref{eq:pqmin} is the projection of a lifted orbit from $\gamma$.  If $(x_k)$ denote the projection of $x_k'$ to $\torus$, we have
  \begin{equation}  \label{eq:zero-action}
    A_c^q(x_0, x_0) + q \alpha(c) = 0.
  \end{equation}
\elem

\noi
\proof
  Let $x' \in \real$ be a lift of $x \in \torus$, then
  \[
    T_c u(x) = \min_{y' \in \real} \left\{ u(y') + H(y', x') - c(x' - y')\right\},
  \]
  where $u$ is treated as a periodic function on $\real$. If $y'_0$ reaches the minimum, then
  \[
    du(y'_0) + \partial_1 H(y'_0, x') + c = 0.
  \]
  This means $F(y'_0, c + du(y'_0)) = (x', p)$ where $p = \partial_2 H(y_0', x')$. Since $\gamma$ is invariant, we must have $p = c + du(x')$. Since $F$ is a diffeomorphism, $y'_0$ is unique. This implies $T_c u$ is differentiable at $x$ and 
  \[
    d(T_c u)(x') = \partial_2 H(y'_0, x') - c = du(x').
  \]
  It follows that $T_c u - u$ is a constant, which necessarily equals $- \alpha(c)$. Note that this argument proves that the minimum in 
  \[
    \min_y \{u(y) + A_c(y, x)\} = T_c u(x) = u(x) - \alpha(c)
  \]
is necessarily achieved at $x_{-1}$, where $F(x_{-1}, c + du(x_{-1}) = (x, du(x))$.

\noi
 Suppose $\gamma$ is $p/q$ periodic. For every $x \in \torus$, let $x_0 = x$, $(x_k, c + du(x_k)) = T^k(x_0, c + du(x_0)) \in \torus \times \real$, and let $x_k'$ denote the projection of the lift of $(x_k, c + du(x_k))$. The argument in the first half of this proof implies $(x_k')_{0 \le k \le q-1}$ is the unique minimizer over $y_k'$ for
  \[
    u(y_0') + \sum_{k = 0} \left( H(y_k', y_{k+1}') - c(y_{k+1}' - y_k')\right), \quad 
    \text{such that } y_q' = x_q'.
  \]
  Therefore it is also the unique minimizer over $(y_k')_{k = 1}^{q-1}$ which satisfies $y_0' = x_0'$ and $y_q' = x_q'$. Under this constraint, $u(y_0')$ and $\sum_{k= 0}^{q-1} c(y_{k+1}' - y_k')$ are constants, therefore $(x_k')$ is the unique minimizer of $\sum_{k = 0}^{p -1}H(y_k', y_{k+1}')$ over the same constraints. This proves both \eqref{eq:pqmin} and its converse.

\noi
  Moreover, we know that  
  \[
      u(x_0) - q \alpha(c)  = T_c^q u(x_0) = u(x_0) + A_c^n(x_0, x_0),
  \]
  this proves \eqref{eq:zero-action}.
\qed

\noi
We define the Peierl's barrier $h_c: \torus \times \torus \to \real$ by
\[
  h_c(x, y) = \liminf_{n \to \infty} \left( A_c^n(x, y) + n \alpha(c) \right).
\]
The \emph{projected Aubry set} is 
\[
  \cA(c) = \{x : \quad h_c(x, x) = 0\}.
\]
The following properties hold for the Peierls barrier.
\protwo{prop:peierls}{See \cite{Fat05, Zav10} }
  \begin{enumerate}[(1)]
    \item $h_c(x, y)$ is Lipschitz in both variables, $h_c(x, x) \ge 0$.
    \item For every $x \in \torus$, $h_c(x, \cdot)$ is a weak KAM solution. 
    \item If $x \in \cA(c)$, then there is a unique $h_c(x, \cdot)$ calibrated orbit ending at $(x, r(x))$. 
    \item (Mather's graph theorem) $r(x)$ is a Lipchitz function over $\cA(c)$.
  \end{enumerate}
\epro

For $x \in \cA(c)$, let $(x, r(x))$ be given by Proposition \ref{prop:peierls}. Then
\[
  \tilde{\cA}(c) = \{(x, r(x)): \, x \in \cA(c)\}
\]
is well defined and is called the Aubry set. It is a compact invariant set under $F$.

\lem{lem:rotation-number}
  For every $x \in \torus$, let $(x_k, r_k)_{k = -\infty}^0$ be any orbit calibrated by $h_c(x, \cdot)$ and $x_k'$ be the lift of the orbit to $\real$. Then the limit 
  \[
    \lim_{k \to -\infty} \frac{x_0' - x_k'}{|k|}
  \]
  exists and depends only on $c$, denoted $\rho(c)$. The map $c \mapsto \rho(c)$ is monotone over all $c$'s for which $\tilde{\cA}(c)$ is a graph over $\torus$.
\elem

\proof
  Let $(x_k')$ corresponds to a lifted orbit, we call $(k, x_k')$ the graph of $(x_k')$. We claim that the graph of any two lifted calibrated orbits $(x_k')$ and $(y_k')$ cannot cross each other in $\real^2$. 

\noi
  Suppose the contrary holds. By the Aubry Crossing Lemma (see \cite{mather1994action}, also note that it only holds in dimension $1$), if any two graphs $(k, x_k')_{k = -n}^0$,  $(k, y_k')_{k = -n}^0$ crosses each other in $\real^2$, then there exists two other sequences $z_k'$, $w_k'$ such that $z_{-k}' = x_{-k}'$, $z_0' = y_0$, $w_{-k} = y_{-k}'$, $w_0 = x_0'$ such that
  \[
    H((z_k')) + H((w_k')) < H((x_k')) + H((y_k')).
  \]
  It follows that
  \[
    A_c^n(z_{-n}, z_0) + A_c^n(w_{-n}, w_0) < A_c^n(x_{-k}, x_0) + A_c^n(y_{-k}, y_0), 
  \]
  where removing $'$ stands for projection to $\torus$. This contradicts the fact that $x_{-k}$ and $y_{-k}$ mninimizes $h_c(y, \cdot) + A_c^n(\cdot, x_0)$ and $h_c(y, \cdot) + A_c^n(\cdot, y_0)$ respectively.

\noi
  If a calibrated orbit does not have a well defined rotation number, then there exists two lifts of it that intersect each other. Similarly, any two orbits with different rotation numbers admit intersecting lifts. This proves $\rho$ is well defined. Moreover, the twist property implies that if $r_2 > r_1$, then the rotation number of the orbit of $(x, r_2)$ is larger than that of $(x, r_1)$. This implies $c \mapsto \rho(c)$ is monotone.
\qed
The following proposition proves (a) $\Rightarrow$ (c) in Proposition \ref{prop:rat-int-twist}. 
\pro{propSanNoM}
Suppose $F$ is rationally integrable in $[\rho_1, \rho_2]$. Then for any $c \in \real$ such that $\rho(c) \in [\rho_1, \rho_2]$, the Aubry set $\tilde{\cA}(c)$ projects onto $\torus$. For every $c \in \torus$, the map $G_x: c \mapsto \tilde{\cA}(c) \cap \pi^{-1}(x)$ is a homeomorphism onto its image.

\noi
  As a corollary, the set in $\torus \times \real$ bounded by the invariant curves of rotation numbers $\rho_1 < \rho_2 \in \rational$ are foliated by the Aubry sets $\tilde{\cA}(c)$, $\rho(c) \in [\rho_1, \rho_2]$.
\epro

\proof
  Suppose $c \in (\rho_1, \rho_2)$, we first show that $\cA(c) = \torus$.
  Given $x \in \torus$, let $n_j \to \infty$ such that $A_c^{n_j}(x, x) \to h_c(x, x)$, and let $(x_k^j, r_k^j)_{k = - n_j}^0$ be minimizers for $A_c^{n_j}(x, x)$. let $(x^\infty_k, r^\infty_k)_{k = -\infty}^0$ be any limit point of the sequence $(x_k^j, r_k^j)$ in $j$ in term wise convergence. We have
  \[
    A_c^{n_j}(x, x_k^j) + (n_j -k) \alpha(c) + A^k(x_k^j, x) + k \alpha(c) = A_c^{n_j}(x, x) + n_j \alpha(c)
  \]
  Taking $\liminf$ to both sides, we get
  \[
    h_c(x, x_k^\infty) + A^k(x_k^\infty, x) + k\alpha(c) \le h_c(x, x).
  \]
  Since the opposite inequality follows directly from definition, we get $x_k^\infty$ is calibrated by $h_c(x, \cdot)$. It follows from Lemma \ref{lem:rotation-number} that $\rho((x_k^\infty)) = \rho(c)$. Since $(x_k^j)$ converges to $(x_k^\infty)$,   then for any lift $(\xi_k^j)$ of $(x_k^j)$, we have
  \[
    \frac{\xi_{n_j}^j - \xi_0^j}{n_j} \to \rho(c) \in (\rho_1, \rho_2).
  \]
  Since the orbit $(x_k^j)$ is minimizing, $(\xi_k^j)$ is a minimizer of \eqref{eq:pqmin} with $p = \xi_{n_j}^j - \xi_0^j$ and $q = n_j$. Lemma \ref{lem:periodic-minimum} then implies the orbit $(x_k^j)$ is contained in some periodic invariant curve $\gamma$ and $A_c^{n_j}(x, x) + n_j \alpha(c) = 0$. Taking $\liminf$, we get
  \[
    h_c(x, x) \le 0.
  \]
  Since $h_c(x, x) \ge 0$, we have $x \in \cA(c)$. Moreover, the argument also proves that for any $(x, r) \in \cA(c)$, there exists a sequence $\gamma_k$ of periodic invariant curves such that $\gamma_k \cap \pi^{-1}(x) \to (x, r)$, where $\pi(x, r) = x$ is the projection. Since all those invariant curves are Aubry sets, this means there exists $c_k$ such that $\rho(c_k) \to \rho(c)$ and $\tilde{\cA}(c_k) \cap \pi^{-1}(x) \to \tilde{\cA}(c) \cap \pi^{-1}(x)$. 

\noi
  Suppose $c_1, c_2 \in \real$ is such that $\tilde{\cA}(c_1) \cap \tilde{\cA}(c_2) \ne \emptyset$. Proposition \ref{prop:peierls} implies that there must exists $x \in \cA(c_1) \cap \cA(c_2)$ such that $u_1 = h_{c_1}(x, \cdot)$ and $u_2 = h_{c_2}(x, \cdot)$ have a common calibrated orbit at $x$. Let $(x_k')_{k \le 0}$ denote the lift of this orbit, we have
  \[
    \alpha(c_1) = \lim_{n \to \infty} \frac{1}{n}H((x_k)_{k = -n}^0) - \frac{c_1 \cdot (x_0' - x_{-n}')}{n}  
    = \lim_{n \to \infty} \frac{1}{n}H((x_k)_{k = -n}^0) - c_1 \cdot \rho((x_k')),
  \]
  where the limit exists due to Lemma \ref{lem:rotation-number}. Apply the same calculation to $c'$ over the same orbit, we get
  \[
  \alpha(c_1) - \alpha(c_2) =  (c_1 - c_2) \cdot \rho((x_k)).
  \]
  The above equality combined with the convexity of $\alpha$ implies $\alpha$ is a linear function on $[c_1, c_2]$. Suppose $(y_k^{(n)}, r_k^{(n)})_{k = 0}^n$ is the lift of a minimizer for $A^n_{(c_1 + c_2)/2}(x, x)$, then  
  \begin{equation}  \label{eq:aubry-action}
    \begin{aligned}
	 & \quad A^n_{(c_1 + c_2)/2}(x, x) + n \alpha\left( \frac{c_1 + c_2}{2}\right)  \\
	 & = H\left( (y_k)_{k = 0}^n \right) - \frac{c_1 + c_2}{2} \cdot (y_n - y_0) + \alpha\left( \frac{c_1 + c_2}{2}\right) \\
	& = \frac{1}{2} \left( H\left( (y_k)\right) - c_1 \cdot (y_n - y_0) + \alpha (c_1) \right)
	 + \frac{1}{2} \left( H\left( (y_k) \right) - c_2 \cdot (y_n - y_0) + \alpha(c_2) \right) \\
	& \ge \frac12 \left( A_{c_1}^n(x, x) + n \alpha(c_1) + A_{c_2}^n(x, x) + n \alpha(c_2) \right)
	\ge 0 .
    \end{aligned}
  \end{equation}
  Suppose $x \in \cA((c_1 + c_2)/2)$, then there exists $n_k \to \infty$ such that $A^{n_k}_{(c_1 + c_2)/2)}(x, x) \to 0$. Then \eqref{eq:aubry-action} implies that the $c_1$ and $c_2$ action of the same orbits also converges to $0$. This implies any limit points of $(x, r_0^{(n_k)})$ converges to a point in $\tilde{\cA}(c_1) \cap \tilde{\cA}(c_2)$. As a result, $\tilde{\cA}((c_1 + c_2)/2) \subset \tilde{\cA}(c_1) \cap \tilde{\cA}(c_2)$. Since all three Aubry set are graphs over $\torus$ as we just proved, we have $\tilde{\cA}(c_1) = \tilde{\cA}(c_2)$. Finally, $c_1 = c_2$ since $c$ is uniquely determined by the invariant graph $\gamma$ via the relation $\gamma = {(x, c + du(x))}$.

\noi
  Since $\psi: c \mapsto \rho(c)$ is monotone, the set $\{c: \, \rho(c) \in [\rho_1, \rho_2]\}$ is an interval $[c_1, c_2]$. The set $\psi([c_1, c_2])$ contains all rational numbers in $[\rho_1, \rho_2]$.  We have proven that the map  $G_x : c \mapsto \tilde{\cA}(c) \cap \pi^{-1} (x)$ is one-to-one over  $[c_1, c_2]$. Moreover, the disjointedness of different $\tilde{\cA}(c)$ implies the map $c \mapsto G_x(c)$ is monotone, and we have proved that $G_x$ is the continuous extension of $G_x|\{c: \, \rho(c) \in \rational\}$, hence continuous. The above argument shows that $G_x$ is a homeomorphism.
\qed

\proof
(Proof of Proposition~\ref{prop:rat-int-twist})
We have already proven (a) $\Rightarrow$ (c). 

To prove (c) $\Rightarrow$ (b), note that by Lemma \ref{lem:periodic-minimum}, every invariant curve of the system is the graph of the gradient of a weak KAM solution, and every orbit on the curve is calibrated by this solution. This implies, in particular, that every orbit is minimizing. It is well known that minimizing orbits don't have conjugate points, see for example \cite{arnaud2010green, contreras1999convex}. This proves (c) $\Rightarrow$ (b).

We now prove (b) $\Rightarrow$ (a). We follow the proof in \cite{AABZ15}, section 2.2. Let $U$ denote the part of the phase space between two invariant curves. For $x \in \torus$, let $V(x) = \{x\} \times [-1, 1]$ be the vertical fiber at $x$. Then the no conjugate point condition imply that the map $F^k$ (lifted to $\real \times [-1, 1]$) is a global diffeomorphism from $V(x) \cap U$ to its image. In particular, the set $F^k(V(x) \cap U) \cap V(x)$ has at most one point. 

Given $x' \in \real$ and $p/q \in [\rho_1, \rho_2]$, define the $(p,q)$-minimal action function
\[
  M_{p,q}(x') = \min\{H((x_k)_{k = 0}^q : \quad x_0 = x', \, x_q = x' + p\}. 
\]
It's known since Birkhoff that the critical points of $M_{p, q}$ correpsonds to periodic orbit. In particular, $\min M_{p, q}$ and $\max M_{p, q}$ corresponds to periodic  orbits. Before proceeding, let's first prove the following:

\textbf{Claim}. If $(x_k)_{k = 1}^q$ is the lift of a periodic orbit outside of $U$, then $(x_q - x_0)/q \notin [\rho_1, \rho_2]$. 

\textbf{Proof of claim}. There are two cases, either $(x_0, r_0)$ is below $\gamma_1$, or it is above $\gamma_2$ (in terms of the $r$ coordinate). We will only prove assuming the former as the other case is similar. Let $(x_k, \delta_k) = V(x_k) \cap \gamma_1$, then by Lemma \ref{twist},
\[
  \pi_1 F^k(x_0, \delta_0)  > x_k
\]
for all $k \ge 0$. It follows that $\frac{x_q - x_0}{q} \le \lim_{k \to \infty} \frac{\pi_1 F^{kq}(x_0, \delta_0)}{kq} = \rho_1$. The inequality is strict since $\gamma_1$ already contain all periodic orbits of rotation number $\rho_1$. This concludes the proof of the Claim.

Continuing with the proof, let $(y_k)_{k = 0}^q$ be the maximizer for $M_{p, q}$. Since $(y_k)$ is periodic, we have $\pi_1 F^{nq}(y_0) = y_0 + np$ for all $n \ge 1$. Our claim implies that the periodic orbit associated with $(y_k)$ is contained in $U$, hence $y_0$ is the unique solution to the equation
\[
\pi_1 F^{nq}(y_0) = y_0 + np.
\]
Moreover, since the orbit realizing $M_{np, nq}$ satisfies the same equation, they must coincide. Let 
\[
  C = \max_{p - 1< x_q - x_0 < p + 1}\min_{x_1, \cdots, x_{q-1}} H((x_k)_{k = 0}^q),
\]
we have
\[
  n H((y_k)_{k = 0}^q) = M_{np, nq}(y_0) \le 2C + (n-2) \min M_{p, q},
\]
because we can move the inner $(n-2)$ cycles to the minimum cycles of $M_{p, q}$, and pay $C$ action connecting $y_0$ to the cycle in the first $q$ steps, and another $C$ action connecting the cycle back to $y_{nq}$. Divide by $n$ and take limit, we get $\max M_{p, q} = H((y_k)_{k = 0}^q) = \min M_{p, q}$, implying $M_{p, q}$ is constant. Since every point $x$ is a critical point of $M_{p, q}$, there exists a invariant curve consisting entirely of $p/q$ periodic orbits. This concludes the proof (b) $\Rightarrow$ (a).
\qed

{\bf Acknowledgment: } We are grateful to the anonymous referee for their careful reading and valuable remarks and comments which helped to improve significantly the paper. V.K. and C.E.K. gratefully acknowledge support from the European Research Council (ERC) through the Advanced Grant ''SPERIG'' (\gwarning{\#885 707}).

\nocite{treschev2013billiard, treschev2015conjugacy, treschev2016locally, wang2022gevrey, callis2022absolutely, huang2018finite}
\bibliographystyle{alpha}
\bibliography{BibtexDatabase}

\begin{thebibliography}{ADSK16}

\bibitem[AABZ15]{AABZ15}
M.~Arcostanzo, M.-C. Arnaud, P.~Bolle, and M.~Zavidovique.
\newblock Tonelli {{Hamiltonians}} without conjugate points and
  ${C}^0-$integrability.
\newblock {\em Mathematische Zeitschrift}, 280(1):165--194, June 2015.

\bibitem[ADSK16]{avila2016integrable}
Artur Avila, Jacopo De~Simoi, and Vadim Kaloshin.
\newblock An integrable deformation of an ellipse of small eccentricity is an
  ellipse.
\newblock {\em Annals of Mathematics}, pages 527--558, 2016.

\bibitem[AMS22]{AMS22}
Marie-Claude Arnaud, Jessica~Elisa Massetti, and Alfonso Sorrentino.
\newblock On the persistence of periodic tori for symplectic twist maps and the
  rigidity of integrable twist maps.
\newblock https://arxiv.org/abs/2202.00313v2, February 2022.

\bibitem[Arn10]{arnaud2010green}
Marie-Claude Arnaud.
\newblock Green bundles and related topics.
\newblock In {\em Proceedings of the International Congress of Mathematicians
  2010 (ICM 2010) (In 4 Volumes) Vol. I: Plenary Lectures and Ceremonies Vols.
  II--IV: Invited Lectures}, pages 1653--1679. World Scientific, 2010.

\bibitem[Bia93]{bialy1993convex}
Misha Bialy.
\newblock Convex billiards and a theorem by e. hopf.
\newblock {\em Mathematische Zeitschrift}, 214(1):147--154, 1993.

\bibitem[Bir13]{birkhoff1913proof}
George~D Birkhoff.
\newblock Proof of poincar{\'e}'s geometric theorem.
\newblock {\em Transactions of the American Mathematical Society}, pages
  14--22, 1913.

\bibitem[Bir22]{Bir22}
George~D. Birkhoff.
\newblock Surface transformations and their dynamical applications.
\newblock {\em Acta Mathematica}, 43(none):1--119, January 1922.

\bibitem[Bir20]{birkhoff2020periodic}
George~D Birkhoff.
\newblock On the periodic motions of dynamical systems.
\newblock In {\em Hamiltonian Dynamical Systems}, pages 154--174. CRC Press,
  2020.

\bibitem[BM17]{bialy2017angular}
Misha Bialy and Andrey~E Mironov.
\newblock Angular billiard and algebraic birkhoff conjecture.
\newblock {\em Advances in Mathematics}, 313:102--126, 2017.

\bibitem[BM22]{bialy2022birkhoff}
Misha Bialy and Andrey~E Mironov.
\newblock The {Birkhoff-Poritsky} conjecture for centrally-symmetric billiard
  tables.
\newblock {\em Annals of Mathematics}, 196(1):389--413, 2022.

\bibitem[Cal22]{callis2022absolutely}
Keagan~G Callis.
\newblock Absolutely periodic billiard orbits of arbitrarily high order.
\newblock {\em arXiv preprint arXiv:2209.11721}, 2022.

\bibitem[CI99]{contreras1999convex}
Gonzalo Contreras and Renato Iturriaga.
\newblock Convex hamiltonians without conjugate points.
\newblock {\em Ergodic Theory and Dynamical Systems}, 19(4):901--952, 1999.

\bibitem[Dou82]{douady1982application}
R~Douady.
\newblock {\em {A}pplications du th{\'e}oreme des tores invariants. {T}h{\`e}se
  de 3{\`e}me cycle}.
\newblock PhD thesis, Universit{\'e} Paris VII, 1982.

\bibitem[Fat05]{Fat05}
Albert Fathi.
\newblock {\em Weak {{KAM}} Theorem in Lagrangian Dynamics Seventh Preliminary
  Version}.
\newblock {Preprint}, 2005.

\bibitem[Gol01]{gole2001symplectic}
Christophe Gol{\'e}.
\newblock {\em Symplectic twist maps: global variational techniques},
  volume~18.
\newblock World Scientific, 2001.

\bibitem[Gut12]{gutkin2012billiard}
Eugene Gutkin.
\newblock Billiard dynamics: An updated survey with the emphasis on open
  problems.
\newblock {\em Chaos: An Interdisciplinary Journal of Nonlinear Science},
  22(2):026116, 2012.

\bibitem[Hal77]{halpern1977strange}
Benjamin Halpern.
\newblock Strange billiard tables.
\newblock {\em Transactions of the American mathematical society},
  232:297--305, 1977.

\bibitem[HF83]{herman1983courbes}
Michael~R Herman and Albert Fathi.
\newblock {\em Sur Les Courbes Invariantes Par Les Diff{\'e}omorphismes de
  l'anneau}, volume~1.
\newblock Soci{\'e}t{\'e} math{\'e}matique de France, 1983.

\bibitem[HK18]{huang2018finite}
Guan Huang and Vadim Kaloshin.
\newblock On the finite dimensionality of integrable deformations of strictly
  convex integrable billiard tables.
\newblock {\em arXiv preprint arXiv:1809.09341}, 2018.

\bibitem[HKS18]{huang2018nearly}
Guan Huang, Vadim Kaloshin, and Alfonso Sorrentino.
\newblock Nearly circular domains which are integrable close to the boundary
  are ellipses.
\newblock {\em Geometric and Functional Analysis}, 28:334--392, 2018.

\bibitem[HZ22]{hezari2022one}
Hamid Hezari and Steve Zelditch.
\newblock One can hear the shape of ellipses of small eccentricity.
\newblock {\em Annals of Mathematics}, 196(3):1083--1134, 2022.

\bibitem[Kov21]{koval2021domains}
Illya Koval.
\newblock Domains which are integrable close to the boundary and close to the
  circular ones are ellipses.
\newblock {\em arXiv preprint arXiv:2111.12171}, 2021.

\bibitem[KS06]{katok2006invariant}
Anatole Katok and Jean-Marie Strelcyn.
\newblock {\em Invariant manifolds, entropy and billiards. Smooth maps with
  singularities}, volume 1222.
\newblock Springer, 2006.

\bibitem[KS18]{kaloshin2018local}
Vadim Kaloshin and Alfonso Sorrentino.
\newblock On the local birkhoff conjecture for convex billiards.
\newblock {\em Annals of Mathematics}, 188(1):315--380, 2018.

\bibitem[MF94]{mather1994action}
John~N Mather and Giovanni Forni.
\newblock Action minimizing orbits in hamiltomian systems.
\newblock In {\em Transition to chaos in classical and quantum mechanics},
  pages 92--186. Springer, 1994.

\bibitem[Res15]{resnikoff2015curves}
Howard~L Resnikoff.
\newblock On curves and surfaces of constant width.
\newblock {\em arXiv preprint arXiv:1504.06733}, 2015.

\bibitem[Sib04]{siburg2004principle}
Karl~Friedrich Siburg.
\newblock {\em The principle of least action in geometry and dynamics}.
\newblock Number 1844. Springer Science \& Business Media, 2004.

\bibitem[Tab05]{tabachnikov2005geometry}
Serge Tabachnikov.
\newblock {\em Geometry and billiards}, volume~30.
\newblock American Mathematical Soc., 2005.

\bibitem[Tre13]{treschev2013billiard}
Dimitry Treschev.
\newblock Billiard map and rigid rotation.
\newblock {\em Physica D: Nonlinear Phenomena}, 255:31--34, 2013.

\bibitem[Tre15]{treschev2015conjugacy}
Dmitry~Valerevich Treschev.
\newblock On a conjugacy problem in billiard dynamics.
\newblock {\em Proceedings of the Steklov Institute of Mathematics},
  289(1):291--299, 2015.

\bibitem[Tre16]{treschev2016locally}
Dmitry Treschev.
\newblock A locally integrable multi-dimensional billiard system.
\newblock {\em arXiv preprint arXiv:1612.00187}, 2016.

\bibitem[WZ22]{wang2022gevrey}
Qun Wang and Ke~Zhang.
\newblock Gevrey regularity for the formally linearizable billiard of treschev.
\newblock {\em arXiv preprint arXiv:2211.03182}, 2022.

\bibitem[Zav10]{Zav10}
Maxime Zavidovique.
\newblock Existence of ${C}^{1,1}$ critical subsolutions in discrete weak
  {{KAM}} theory.
\newblock {\em Journal of Modern Dynamics}, 4(4):693--714, Fri Dec 31 19:00:00
  EST 2010.

\end{thebibliography}
\end{document}